\documentclass[11pt, reqno]{amsart}
\usepackage{diagbox}
\usepackage{color}
\usepackage{amsmath}
\usepackage{latexsym,amssymb}
\usepackage[mathscr]{euscript}
\usepackage{graphics,graphicx,subfig,float}
\usepackage{stmaryrd}
\usepackage{bm}
\usepackage{enumerate}
\usepackage{setspace}
\newtheorem{thm}{Theorem}[section]
\newtheorem{proposition}{Proposition}[section]

\newtheorem{remark}{\bf Remark}[section]
\newtheorem{lemma}{\bf Lemma}[section]

\numberwithin{equation}{section}
\newtheorem{assumption}{Assumption}[section]

\begin{document}
		\baselineskip=17pt
	\title{OPTIMAL CONTROL FOR PRODUCTION INVENTORY SYSTEM WITH VARIOUS COST CRITERION}

		\author[Subrata Golui]{Subrata Golui}
	\address{Department of Mathematics\\
		Indian Institute of Technology Guwahati\\
		Guwahati, Assam, India}
	\email{golui@iitg.ac.in}
	
	\author[Chandan Pal]{Chandan Pal}
	\address{Department of Mathematics\\
		Indian Institute of Technology Guwahati\\
		Guwahati, Assam, India}
	\email{cpal@iitg.ac.in}

\author[Manikandan, R.]{Manikandan, R.}
\address{Department of Mathematics\\
Central University of Kerala\\
Kasaragod, Kerala, India}
\email{mani@cukerala.ac.in}

\author[Abhay Sobhanan]{Abhay Sobhanan}
\address{Department of Industrial and Management Systems Engineering\\
University of South Florida\\
Tampa, FL, 33620, USA}
\email{sobhanan@usf.edu}
	\date{}
	
		\maketitle \baselineskip17pt
	\parskip10pt
	\parindent.4in
	\begin{abstract}
		\vspace{2mm}
		\noindent 
		In this article, we investigate a dynamic control problem of a production-inventory system. Here, demands arrive at the production unit according to a Poisson process and are processed in an FCFS manner. The processing time of the customer’s demand is exponentially distributed. The production manufacturers produce the items on a make-to-order basis to meet customer demands. The production is run until the inventory level becomes sufficiently large. We assume that the production time of an item follows exponential distribution and the amount of time for the produced item to reach the retail shop is negligible. Also, we assume that no new customer joins the queue when there is a void inventory. This yields an explicit product-form solution for the steady-state probability vector of the system. The optimal policy that minimizes the discounted/average/pathwise average total cost per production is derived using a Markov decision process approach. We find optimal policy using value/policy iteration algorithms. Numerical examples are discussed to verify the proposed algorithms.
	\end{abstract}

	\maketitle
	\noindent {\bf Key Words:} Production inventory system, controlled Markov chain, cost criterion, value iteration algorithm, policy iteration algorithm.
	
	\noindent {\bf Mathematics Subject Classification}:  Primary 93E20, Secondary 49L20, 60J27.
	\section{Introduction}
	Inventory theory has useful applications in various day-to-day real-life scenarios. One such application is production control, in which decision-makers focus on controlling costs while satisfying customer demands and maintaining their goodwill. Over the last decade research on complex integrated production-inventory systems or service-inventory systems has found much attention, often in connection with the research on integrated supply chain management, see He et al. (2002); He and Jewkes (2000); Helmes et al. (2015); Krishnamoorthy et al. (2015); Krishnamoorthy and Narayanan (2013); Malini and Shajin (2020); Pal et al. (2012);  Sarkar (2012); Veatch and Wein (1994). In these articles the authors considered $(s,S)/(s,Q)$-{type} policy to study their inventory models.
	
	Sigman and Simchi-Levi (1992) and  Melikov and Molchanov (1992) introduced the integrated queueing-inventory models. Whereas the article by Sigman and Simchi-Levi (1992), considered the Poisson arrival of demands, arbitrarily distributed service time, and exponentially distributed replenishment lead time. Also, they showed that the resulting queueing-inventory system is stable if and only if the service rate is higher than the customer arrival rate. The authors considered that the customers may join the system even when the inventory level is zero and discussed the case of non-exponential lead-time distribution. Berman et al. (1993) followed them with deterministic service times and formulated the model as a dynamic programming problem. For more inventory models with positive service times, see Berman and Kim (1999), (2004); Arivarignan et al. (2002); Krishnamoorthy et al. (2006a), (2006b); for a recent extensive survey of literature, we refer in Krishnamoorthy et al. (2021), it provides the summary of work done until 2019.  
	
	We recall the remarkable work by Schwarz et al. (2006). They propose product form solutions for the system state distribution under the assumption that customers do not join when the inventory level is zero, where the service/lead time is exponentially distributed and demands follow a Poisson distribution. Krishnamoorthy and Narayanan (2013) reduced the Schwarz et al. (2006) model to a production inventory system with single-batch bulk production of the quantum of inventory required. The production inventory with service time and protection for a few of the final phases of production and service is discussed in Sajeev (2012).
	
	Saffari et al. (2011) considered an M/M/1 queue with inventoried items for service, where the control policy followed is $(s, Q)$ and the lead time is a mixed exponential distribution. They assumed that when inventory stock is empty, fresh arrivals are lost to the system, and thus, they obtain a product form solution for the system state probability. Inventory system with queueing networks was studied by Schwarz et al. (2007). The authors assumed that at each service station, an order for replenishment is made when the inventory level at that station drops to its reorder level; hence, no customer is lost to the system.  Zhao and Lian (2011) used  dynamic programming to obtain the necessary and sufficient conditions for a priority queueing inventory system to be stable. 
	
	In all the papers quoted above, customers are provided with an item from the inventory after their completion of service. In Krishnamoorthy et al. (2015), customers may not get an inventory after their completion of service. They studied the optimization problem and obtained the optimal pairs $(s,S)$ and $(s, Q)$ corresponding to the expected minimum costs.
	
	In this study, we do not use any common inventory control policies such as $(s,S)/ (s,Q)$-{type}. We consider the problem of finding the optimal production rates for a discounted/long-run average/pathwise average cost criterion of the production inventory system. Here, we consider an $M/M/1/ \infty$ production inventory system with positive service time. Customers’ demands arrive one at a time according to a Poisson processes. Service and production times follow an exponential distribution. Each production is $1$ units, and the production process is run until the inventory level becomes sufficiently large (infinity). It is assumed that the amount of time for the item produced to reach the retail shop is negligible. 
	We assume that no customer joins a queue when the inventory level is zero. This assumption leads to an explicit product-form solution for the steady-state probability vector using a simple approach. In this paper, we have applied matrix analytic methods for finding system steady-state equations. Readers are referred to Neuts (1989), (1994), Chakravarthy and Alfa (1986) and Chakravarthy (2022a and 2022b). 
	
	In this paper, we find an optimal stationary policy by policy/value iteration algorithm. We see that there are many studies on inventory production control theory on continuous-time controlled Markov decision processes (CTCMDPs) for discounted/ average/ pathwise average cost criteria (see Federgruen and Zipkin (1986), (1986); Helmes et al. (2015)). However, the  articles discussed on algorithms for finding an optimal stationary policy are, Federgruen and Zhang (1992); He et al. (2002); He and Jewkes (2000). The fixed costs of ordering items or setting up a production process arise in many real-life scenarios. In their presence, the most widely used ordering policy in the stochastic inventory literature is the $(s,S)$ policy. In this context, we mention two important survey papers for discrete/continuous-time regarding $(s,S)$ replenishment policy: Perera and Sethi (2022a, 2022b). They comprehensively surveyed the vast literature accumulated over seven decades in these two papers for the discounted/average cost criterion on discrete/continuous-time.
	
	The motivation for studying discounted problems comes mainly from economics. For instance, if $\delta$ denotes a rate of discount, then $(1 +\delta)L$
	would be the amount of money one would have to pay to obtain a loan
	of $L$ dollars over a single period. Similarly, the value of a note promising to
	pay $L$ dollars $t$ time steps into the future would have a present value of
	$\frac{L}{(1+\delta)^t}=\alpha^tL$, where $\alpha:=(1+\delta)^{-1}$ denotes the discount factor.
	This is the case for finite-horizon problems. But in some cases,
	for instance, processes of capital accumulation for an economy, or some
	problems on inventory or portfolio management, do not necessarily have a
	natural stopping time in the definable future, see Hern$\acute{\rm a}$ndez-Lerma and Lasserre (1996); Puterman (1994).
	Now when decisions are made frequently, so that the discount rate is very close to 1, or
	when performance criterion cannot easily be described in economic terms, the
	decision maker may prefer to compare policies on the basis of their average expected
	reward instead of their expected total discounted reward, see Piunovsky and Zhang (2020).\\
	
	The ergodic problem for controlled Markov processes refers to the problem of minimizing the time-average cost over an infinite time horizon. Hence, the cost over any finite initial time segment does not affect ergodic cost. This makes the analytical analysis of ergodic problems more difficult. However, the sample-path cost $r(\cdot,\cdot,\cdot)$, defined by \eqref{cost}, corresponding to an average-cost optimal policy that minimizes the expected average cost may fluctuate from its expected value. To take these fluctuations into account, we next consider the pathwise average-reward (PAC) criterion.
	In this study, we investigate the production inventory control problem for the discounted/average/pathwise average cost criterion. We find the optimal production rate through a value/policy iteration algorithm. However, there may be some issues in obtaining an optimal policy. Hence, in this study, we also examine an $\varepsilon$-optimal policy. Finally, numerical examples are included to verify the proposed algorithms. 
	
The remainder of this paper is organized as follows. First, we define the production control problem in section 2. In Section 3, we discuss the steady-state analysis of this model and describe the evaluation of the control system. In addition, we define our cost criterion and assumptions required to obtain an optimal policy. Section 4 discusses the discount cost criterion. Here, we find a solution for the optimality equation corresponding to the discounted cost criterion, and provide its value/policy iteration algorithms. In the next section, we deal with the optimality equation and policy iteration algorithm corresponding to the average cost criterion. We perform the same analysis in Section 6 for the pathwise average cost criterion, as in Section 5. Finally, in Section 7, we provide concluding remarks and highlight the directions for future research.
	
	\noindent \textbf{Notations:}\\
	\noindent $\mathcal{N}(t)$: number of customers in the system at time $t$.\\
	$\mathcal{I}(t)$: inventory level in the system at time $t$.\\
	$e: (1, 1, \cdots, 1,\cdots)$  a column vector of $1^{'}s$ of appropriate order.\\
	(LI)QBD: (Level independent) Quasi birth and death process.\\
	$\mathbb{N}_0=\mathbb{N}\cup \{0\}$, where $\mathbb{N}$ is set of all natural numbers.\\
	$C_b(\mathbb{N}_0\times \mathbb{N}_0)$ is the collection of all bounded functions on $\mathbb{N}_0\times \mathbb{N}_0$.
	
	\section{Problem Description}
	\subsection{Production inventory model}
	We consider an $M/M/1/ \infty$ production inventory system with positive service time. Demands by customers for the item occur according to a Poisson process of rate $\lambda$. Processing of the customer request requires a
	random amount of time, which is exponentially distributed with parameter $\mu$. Each production is of $1$ unit and the production process is keep \textit{{run}} until inventory level becomes sufficiently large (infinity). To produce an item it takes an amount of time which is exponentially distributed with parameter $\beta$. We assume that no customer is allowed to join the queue when the inventory level is zero; such demands are
	considered as lost. It is assumed that the amount of time for the item
	produced to reach the retail shop is negligible. Thus the system is a continuous-time Markov chain (CTMC) $\left\{\mathcal{X}(t);t\geq 0\right\}=\left\{\left(\mathcal{N}(t),\mathcal{I}(t)\right); t\geq 0\right\}$ with state space \mbox{\boldmath {{$\Omega$}}}$=\bigcup \limits_{n = 0}^\infty  {\mathcal{L}(n)},$ where $\mathcal{L}(n)$ is called the $n^{th}$ level of the CTMC, is given by, $\left\{(n,i); i \in \mathbb{N}_0\right\}.$
	
Now the transition rates in the CTMC are:
\begin{itemize}
	\item  $ (n,i)\rightarrow(n+1,i)$   \quad \quad \qquad : rate is $\lambda$, \quad $n\in \mathbb{N}_0$, $i\in \mathbb{N}$
	\item  $ (n,i)\rightarrow(n-1,i-1)$ \quad \qquad: rate is $\mu$, \quad $n,i\in \mathbb{N}$
	\item $ (n,i)\rightarrow(n,i+1)$    \quad \qquad \quad : rate is $\beta$, \quad $n,i\in \mathbb{N}_0.$
	\item All other transition rates are zero.
	\end{itemize} 
Write, \begin{eqnarray*}
		P\{N(t)=n, I(t)=i\} = P_{n, i}(t).
	\end{eqnarray*}
~\\These satisfies the system of difference-differential equations:
	\begin{eqnarray}
	P^{'}_{n,0}(t) = - \beta P_{n,0}(t) + \mu P_{n+1,1}(t) , \quad n\in \mathbb{N}_0 &\\
	P^{'}_{n,i}(t) = - (\lambda+\beta+\mu) P_{n,i}(t) + \mu P_{n+1,i+1}(t) + \lambda P_{n-1,i-1}(t) + \beta P_{n, i+1},\quad n,i \in \mathbb{N}.
\end{eqnarray}
~\\ The steady-state time derivative is equated to zero under the condition for its existence is $\lambda<\mu$, which will be proved in the subsequent Lemma 3.1.\\ Write, 
\begin{eqnarray*}
 	\lim_{t\rightarrow\infty} P_{n, i}(t)= P_{n, i}, \quad n,i \in \mathbb{N}.
 \end{eqnarray*}
Thus, the above set of equations (2.1) and (2.2) becomes,
	\begin{eqnarray}
	-\beta P_{n,0}+ \mu P_{n+1,1} =0, \quad n \in \mathbb{N}_0 \quad & \\
	-(\lambda+\beta+\mu)P_{n,i} + \mu P_{n+1,i+1} + \lambda P_{n-1,i-1} + \beta P_{n, i+1} = 0, \quad n, i \in \mathbb{N}.
\end{eqnarray}

We can solve these equations to find the steady-state solution, by using the Matrix-Analytic Method.

\begin{figure}[H]
		\centering
		\includegraphics[width=0.9\textwidth]{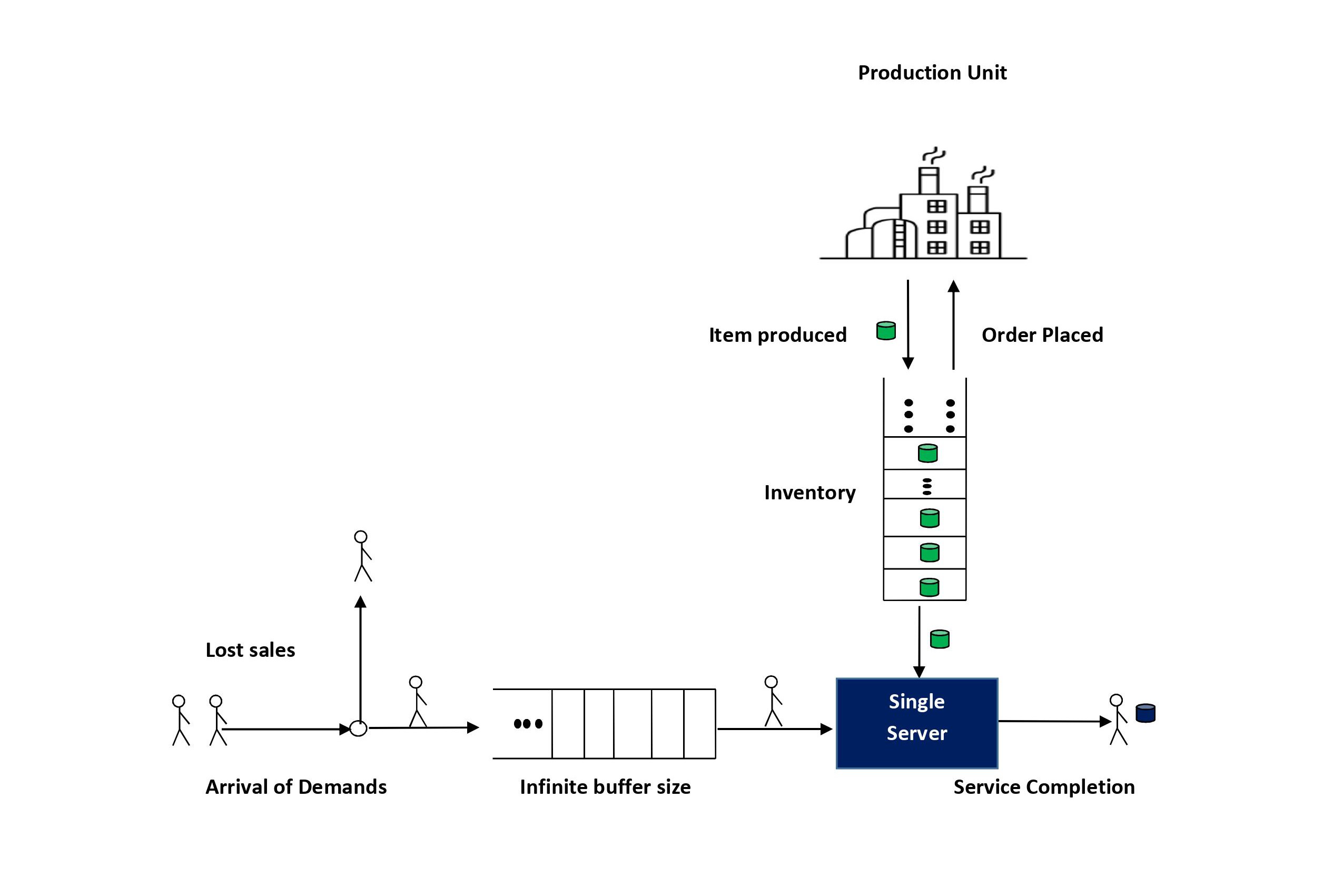}
		\caption{Dynamic production inventory system.}
		\label{figure0}
	\end{figure}

	The infinitesimal generator of this CTMC $\left\{\mathcal{X}(t);t\geq 0\right\}$ is\\  
	\begin{equation} \label{eq:1}
	\mbox{\boldmath {{$\mathcal{Q}$}}}=
	\left[ {\begin{array}{*{20}c}
		{B } & {A_0 } & {} & {} & {} & {}  \\
		{A_2 } & {A_1 } & {A_0 } & {} & {} & {}  \\
		{} & {} {A_2 } & {A_1 } & {A_0 } & \dots  \\
		{} & {} & {} {\ddots} & {\ddots} & {\ddots}  \\
		\end{array}} \right],
	\end{equation}
	where $B$ contains transition rates within $\mathcal{L}(0)$; $A_0$ is the arrival matrix that represents the transition rates of customer arrival i.e., $A_0$ represents the transition from level $n$ to level $n+1,$ for any $ n \in \mathbb{N}_0$; $A_1$ represents the transitions within $\mathcal{L}(n)$ for any $n \in \mathbb{N}$ and $A_2$ is the service matrix that represents the transition rates of service times i.e., $A_2$ represents transitions from $\mathcal{L}(n)$ to $\mathcal{L}(n-1), \ n \in \mathbb{N}.$ The transition rates are
	$$
	\left[B\right]_{kl} =\left \{
	\begin{array}{ll}
	-\beta,             & \textrm{for  }l= k = 0,   \\
	-(\lambda+\beta),    & \textrm{for  }l = k ; \ k=1, 2, ..., \infty,\\
	\beta,               & \textrm{for  }l = k+1; \ k=0, 1, ..., \infty,\\
	0,                  & \textrm{otherwise,}\\
	\end{array}
	\right.
	$$
	$$
	\left[A_{0}\right]_{kl} =\left \{
	\begin{array}{ll}
	\lambda,            & \textrm{for  }l = k ; \ k=1, 2, ..., \infty,\\
	0,            & \textrm{otherwise,}\\
	\end{array}
	\right.
	$$
	$$
	\left[A_{1}\right]_{kl} =\left \{
	\begin{array}{ll}
	-\beta,                 & \textrm{for  }l = k = 0,   \\
	-(\lambda+\beta+\mu),    & \textrm{for  }l = k; \ k = 1, 2, ...,\infty, \\
	\beta,               & \textrm{for  }l = k+1; \ k = 0,1,..., \infty,\\
	0,            & \textrm{otherwise,}\\
	\end{array}
	\right.
	$$
	$$
	\left[A_{2}\right]_{kl} =\left \{
	\begin{array}{ll}
	\mu,             & \textrm{for  }l = k-1 ; \ k = 1, 2, ..., \infty,\\
	0,            & \textrm{otherwise.}\\
	\end{array}
	\right.
	$$
	All other remaining transition rates are zero.\\
	Note: All entries (block matrices) in \mbox{\boldmath {{$\mathcal{Q}$}}} have infinite order, and these matrices contain transition rates within level (in the case of diagonal entries) and between levels (in the case of off diagonal entries).

	\section{Analysis of the system} \label{sc3}
	
	In this section we carry out the steady-state analysis of the production inventory model under study by first establishing the stability condition of the system. Define  \mbox{\boldmath {{${A}$}}}=$A_0+A_1+A_2$. This is the infinitesimal generator of the infinite state CTMC corresponding to the inventory level $\left\{0, 1, \dots, \infty \right\}$. Let {\boldmath$\varphi$} denote the steady-state probability vector of \mbox{\boldmath {{${A}$}}}. That is \mbox{\boldmath{{$\varphi$}}} satisfies  
	\begin{equation} \label{eq:2}
	\mbox{\boldmath{{$\varphi A$}}}=0,\ \mbox{\boldmath {{$\varphi$}}\textbf{e}}=1, ~\text{where}~ e~\text{is a column vector of ones}.
	\end{equation}
	Write $$\mbox{\boldmath{{$\varphi$}}}=\left(\varphi_{0}, \varphi_{1}, \varphi_{2}, \dots \right).$$
	We have
	\begin{center} 
		\mbox{\boldmath {{${A}$}}} =$
		\bordermatrix{& & &    &  &  & & \cr
			&-\beta    &\beta &  &   &   &   \cr
			& \mu&-(\mu+\beta) &\beta   &       \cr 
			& &\mu&-(\mu+\beta) & \beta    &     \cr
			&&&\ddots&\ddots &\ddots &     \cr       
		}
		$.
	\end{center}
	Then using \eqref{eq:2} we get the components of the probability vector \mbox{\boldmath {{$\varphi$}}} (note that, $\beta< \mu$)  explicitly as: 
	$$\varphi_{n}={\left( 1-{\frac{{ \beta }}{{ \mu }}} \right) }{\left({\frac{{\beta}}{{ \mu }}} \right)^{n} }, n=0,1,\dots, \infty.$$
	Since the Markov chain under study is an level independent quasi Birth-Death (LIQBD) process, it is stable if and only if the left drift rate exceeds the right drift rate. 
	That is,  
	\begin{equation} \label{ep2}
	\mbox{\boldmath {{$\varphi$}}} A_0\textbf{e}<\mbox{\boldmath {{$\varphi$}}} A_2\textbf{e}.
	\end{equation}
	We have the following lemma:
	\begin{lemma}
		The stability condition of the production inventory model is given by $\lambda<\mu$.
	\end{lemma}
	\begin{proof}
		
		From the well known result in Neuts (1994) on the positive recurrence of $A$, we have\\ 
		\mbox{\boldmath {{$\varphi$}}} $A_0\textbf{e}< \mbox{\boldmath {{$\varphi$}}} A_2\textbf{e}.$
		With a bit of computation, this simplifies to the result $\lambda<\mu$. 
	\end{proof}
	For future reference we define $\rho$ as 
	\begin{equation} \label{ep3}
	\rho=\frac{\lambda}{\mu}.
	\end{equation}

	\subsection{Steady-state analysis}
	For computing the invariant measure of the process $\{\mathcal{X}(t);t\geq0\}$, we first consider a production inventory system with negligible service time where no backlog of customers is allowed (that is when inventory level is zero, no demand joins the system). The rest of the assumptions such as those on the arrival process and lead time are the same as given earlier. Designate the Markov chain so obtained as $\{\widetilde{\mathcal{X}}(t);t\geq 0\}$=$\{\left(\mathcal{I}(t)\right);t\geq0\}$. Its infinitesimal generator \mbox{\boldmath {{$\widetilde{\mathcal{Q}}$}}} is given by,\\
	
	\begin{center}
		\mbox{\boldmath {{$\widetilde{\mathcal{Q}}$}}} =$
		\bordermatrix{& & &    &  &  & & \cr
			&-\beta   &\beta &  &   &   &   \cr
			& \lambda&-(\lambda+\beta) & \beta   &       \cr 
			& &\lambda&-(\lambda+\beta) & \beta   &     \cr
			&&&\ddots&\ddots &\ddots &     \cr       
		}
		$.
	\end{center}  
	Let \mbox{\boldmath {{$\pi$}}}=$(\pi_{0}, \pi_{1},\pi_{2}, \dots)$ be the invariant measure of the process $\widetilde{\mathcal{X}}(t)$=$\{\mathcal{I}(t);t\geq0\}$. Then \mbox{\boldmath {{$\pi$}}} satisfies the relations
	\begin{equation} \label{ep4}
	\mbox{\boldmath {{$\pi\widetilde{\mathcal{Q}}$}}}=0, \ \mbox{\boldmath {{$\pi$}}}\textbf{e}=1. 
	\end{equation}
	That is,
	at arbitrary epochs of the components of the inventory level probability distribution \mbox{\boldmath {{$\pi$}}} (with $\lambda\neq\beta$) is given by:
	
	\begin{equation}
	\pi_{i}={\left( 1-\frac{\beta}{\lambda} \right) }{\left(\frac{\beta}{\lambda} \right)^{i} }, i=0,1,\dots, \infty. 
	\label{eq:5}
	\end{equation}
	Using the components of the probability vector \mbox{\boldmath {{$\pi$}}}, we shall find the invariant measure of the CTMC $\left\{\mathcal{X}(t);t\geq 0\right\}$. For this, let \mbox{\boldmath {{$P$}}} be the invariant measure of the original system. Then the invariant probability vector must satisfy the set of equations
	\begin{equation} \label{ep5}
	\mbox{\boldmath {{$P\mathcal{Q}$}}}=0,\ \mbox{\boldmath {{$P$}}}\textbf{e}=1.
	\end{equation}
	Partition \mbox{\boldmath {{$P$}}} by levels as  
	\begin{equation} \label{ep6}
	\mbox{\boldmath {{$P=(P_0,P_1,P_2,\dots)$}}},
	\end{equation}
	where the sub vectors of \mbox{\boldmath {{$P$}}} are further partitioned as,
	$$
	\mbox{\boldmath {{$P_n$}}}=(P_{n,0}, P_{n,1}, P_{n,2}, \dots), \quad n \in \mathbb{N}_0.
	$$
	Then the above system of equations reduces to
	\begin{equation} \label{ep7}
	\mbox{\boldmath {{$P_0$}}}B+\mbox{\boldmath {{$P_1$}}}A_2=0.
	\end{equation}
	\begin{equation} \label{ep8}
	\mbox{\boldmath {{$P_n$}}}A_0+\mbox{\boldmath {{$P_{n+1}$}}}A_1+\mbox{\boldmath {{$P_{n+2}$}}}A_2=0, \quad n \in \mathbb{N}_0.
	\end{equation}
	Assume that
		\begin{equation} \label{ep9}
	{\mbox{\boldmath {{$P_0$}}}=\xi\mbox{\boldmath {{$\pi$}}}.}
	\end{equation}
	\begin{equation} \label{ep10}
	\mbox{\boldmath {{$P_n$}}}=\xi\left(
	\frac{\lambda }{\mu }
	\right)^{n}\mbox{\boldmath {{$\pi$}}}, n\geq 1,
	\end{equation}
	where $\xi$ is a constant to be determined. We verify that the equations \eqref{ep7} and \eqref{ep8} are satisfied by \eqref{ep9} and \eqref{ep10}. From \eqref{ep7}, we have
	\begin{equation} \label{ep11}
	\mbox{\boldmath {{$P_0$}}}B+\mbox{\boldmath {{$P_1$}}}A_2= \xi \mbox{\boldmath {{$\pi$}}}\left(B+\frac{\lambda }{\mu } A_2\right), 
	\end{equation}
	and from relation \eqref{ep8}, we have,
	\begin{equation} \label{ep12}
	\mbox{\boldmath {{$P_n$}}}A_0+\mbox{\boldmath {{$P_{n+1}$}}}A_1+\mbox{\boldmath {{$P_{n+2}$}}}A_2 = \xi\left( {\frac{\lambda }{\mu }} \right)^{n + 1} \mbox{\boldmath {{$\pi$}}} \left( {B + \frac{\lambda }{\mu }A_2 } \right). 
	\end{equation}
	Now from the matrices $B, A_2$ and \mbox{\boldmath {{$\widetilde{\mathcal{Q}}$}}}, it follows that
	\begin{equation} \label{ep13}
	B + \frac{\lambda }{\mu }A_2 =  \mbox{\boldmath {{$\widetilde{\mathcal{Q}}$}}}.
	\end{equation}
	Also from \eqref{ep4} we have $ \mbox{\boldmath {{$\pi\widetilde{\mathcal{Q}}$}}}=0$. Hence the right hand side of the equation \eqref{ep11} and \eqref{ep12} are zero. Hence if we take the vector \mbox{\boldmath {{$P$}}} as given by \eqref{ep6}, it follows that \eqref{ep7} and \eqref{ep8}  are satisfied. Now applying the normalizing condition \mbox{\boldmath {{$P$}}}\textbf{e}=1, we get $$
	\xi \left[ {1 + \frac{\lambda }{\mu } + \left( {\frac{\lambda }{\mu }} \right)^2  + \left( {\frac{\lambda }{\mu }} \right)^3  +  \cdots } \right] = 1.
	$$
	Hence under the condition that $\lambda<\mu$, we have 
	\begin{equation} \label{ep14}
	\xi = 1 -  {\frac{\lambda }{\mu }},	\end{equation}
	for more details, see Krishnamoorthy et al. (2015); Malini and Shajin (2020).
	\begin{thm}
		Under the necessary and sufficient condition $\lambda <\mu$ for stability, the components of the invariant measure of the process $\left\{\mathcal{X}(t);t\geq0\right\}$ with generator matrix \mbox{\boldmath {{$\mathcal{Q},$}}} is given by \eqref{ep9}, \eqref{ep10} and \eqref{ep14}. That is, $\mbox{\boldmath {{$P_0$}}}=(1-\rho)\mbox{\boldmath {{$\pi$}}}, \mbox{\boldmath {{$P_n$}}}= (1-\rho) \rho^{n} \mbox{\boldmath {{$\pi$}}},$ for $n \in \mathbb{N}$, where $\rho$ is as defined in \eqref{ep3} and  \mbox{\boldmath {{$\pi$}}} is the inventory level probability vector as given in \eqref{eq:5}.
	\end{thm}

	It is very natural to assume that our production rate function never goes to zero because of heavy starting cost and at any time it depends on the number of inventory and the number of customer in the queue, i.e., it  is a map $${\beta}: \mathbb{N}_0\times \mathbb{N}_0 \to [\gamma,R],$$
	where $\mathbb{N}_0:=\{0,1,2,\cdots \}$, and $\gamma, R$ are some positive constant. Here in our model, state space is \mbox{\boldmath {{$\Omega$}}}$=\bigcup \limits_{n = 0}^\infty \left\{(n,i);  i \in \mathbb{N}_0\right\} $, and the action space is $A=[\gamma,R]$ also let for any state $(n,i)\in \mathbb{N}_0\times \mathbb{N}_0$, the corresponding admissible action space is $A(n,i)=[\gamma,R]$. Now, consider a Borel subset of $\mathbb{N}_0\times \mathbb{N}_0\times [\gamma,R]$ denoted by $K:=\{(n,i, \tilde{\beta})|n\in \mathbb{N}_0,i\in \mathbb{N}_0, \tilde{\beta}\in [\gamma,R]\}$. Recall (p. 6-7) corresponding to state $(n,i)$ and $\tilde{\beta}\in [\gamma,R]$, we denote the transition rates as $\pi^{\tilde{\beta}}_{(n,i)(m,j)}$.
	\begin{align}
	\left\{ \begin{array}{ll}&\pi^{\tilde{\beta}}_{(0,i)(0,j)}=(B)_{ij}\\
	&\pi^{\tilde{\beta}}_{(n,i)(n+1,j)}=(A_0)_{ij}~\text{ for any}~n\in \mathbb{N}_0\\
	&\pi^{\tilde{\beta}}_{(n,i)(n,j)}=(A_1)_{ij}~\text{ for any}~n\in \mathbb{N}\\
	&\pi^{\tilde{\beta}}_{(n,i)(n-1,j)}=(A_2)_{ij}~\text{ for any}~n\in \mathbb{N}.\label{eq 3.1}
	\end{array}\right.
	\end{align}
	All other transition rates are zero.
	Note that, 	\begin{align}\label{EQ2.1}
	&\sum_{(m,j)\in \mathbb{N}_0\times\mathbb{N}_0}\pi^{\tilde{\beta}}_{(n,i)(m,j)}\equiv 0,~\forall (n,i, \tilde{\beta})\in K
	\end{align} 
	and
	\begin{align}\label{E5.16}
	\displaystyle{\sup_{(n,i) \in \mathbb{N}_0 \times \mathbb{N}_0 }}\pi^*_{(n,i)}={\sup_{(n,i) \in \mathbb{N}_0 \times \mathbb{N}_0 }}\sup_{\tilde{\beta}\in [\gamma,R]} \pi^{\tilde{\beta}}_{(n,i)}={\sup_{(n,i) \in \mathbb{N}_0 \times \mathbb{N}_0 }} \sup_{\tilde{\beta}\in [\gamma,R]} \Big [ -\pi_{(n,i)(n,i)}^{\tilde{\beta}} \Big ]=R+\mu+ \lambda<\infty .
	\end{align} 
	Define $r(n,i,\beta)$ is the cost function in the long run corresponding to production rate function $\beta$. Then the cost function is of the form:
	\begin{equation}\label{cost}
	r(n,i,\beta)=h\cdot i+c_1\cdot n+\beta\cdot c_2\cdot I_{\{i>S\}}+c_3\cdot n\cdot I_{\{i=0\}},
	\end{equation}
	where $h$ is the holding cost per item per unit time in the ware house, $c_1$ is the service cost per customer, $c_2$ is the storage/penalty cost per item per production when the inventory level is beyond $S$ and $c_3$ is the cost incurred due to loss per customer when the item of the inventory is out of stock. Note that our cost function is continuous in the third argument for each fixed first $(n,i) \in \mathbb{N}_0\times \mathbb{N}_0 $. Here our aim is to minimize our accumulated cost over all production rate functions, i.e.,   
	$$\mathscr{U}_{SM}:=\{{\beta} \; | \; {\beta}: \mathbb{N}_0\times \mathbb{N}_0 \to [\gamma,R]\}.$$
	This is the collection of all deterministic stationary strategies/policies. Note that we can write $\mathscr{U}_{SM}$ as the countable product space $[\gamma,R]$. So, Tychonoff's theorem (see [Guo and Hern$\acute{\rm a}$ndez-Lerma (2009), Proposition A. 6]) yields that $\mathscr{U}_{SM}$ is compact.\\
	{\bf Evolution of the Control System:} 	Next, we give an informal description of the evolution of the CTCMCs as follows. The controller observes continuously the current state of the system. When the system is in state $(n,i)\in $\mbox{\boldmath {{$\Omega$}}} at time $t\geq0$, he/she chooses action $\tilde{\beta}\in [\gamma,R]$ according to some control. As a consequence of this, the following happens: 
	\begin{itemize}
		\item the controller incurs an immediate  cost at rate $r(n,i, \tilde{\beta})$; and\\
		\item the system stays in state $(n,i)$ for a random time, with rate of leaving $(n,i)$ given by $\pi_{(n,i)}^{\tilde{\beta}}$, and then jumps to a new state $(m,j)\neq (n,i)$ with the probability determined by $\dfrac{\pi_{(n,i)(n,i)}^{\tilde{\beta}}}{\pi_{(n,i)}^{\tilde{\beta}}}$ (see [Guo and Hern$\acute{\rm a}$ndez-Lerma (2009), Proposition B.8] for details).
	\end{itemize} 
	When the state of the system transits to the new state $(m,j)$, the above procedure is repeated.
	The controller tries to minimize his/her costs with respect to some performance criterion defined by (\ref{main}), (\ref{ergodicriskcost}) and (\ref{EQ9.1}) below.\\
	For each ${\beta}\in \mathscr{U}_{SM}$, the associated rates are defined as
	\begin{align}\label{E2.2}
	\pi_{(n,i)(m,j)}^{\beta}:=\pi_{(n,i)(m,j)}^{{\beta}(n,i)}~\text{ for }~(n,i),(m,j)\in \mathbb{N}_0\times\mathbb{N}_0~\text{for}~t\geq 0.
	\end{align} 
	Let $Q({\beta}):=\left[\pi_{(n,i)(m,j)}^{{\beta}}\right]$ be the associated matrix of transition rates with the $((n,i),(m,j))^\text{th}$ element $\pi_{(n,i)(m,j)}^{{\beta}}$. Any (possible substochastic and homogeneous) transition function\\ $\tilde{p}(s,(n,i),t,(m,j),{\beta})$ such that
	$$\lim_{\gamma\rightarrow 0^{+}}\frac{\tilde{p}(t,(n,i),t+\gamma,(m,j),{\beta})-\delta_{(n,i)(m,j)}}{\gamma}=\pi_{(n,i)(m,j)}^{{\beta}}$$ is called a $Q$-processes with the transition rate matrices $Q({\beta})$, where $\delta_{(n,i)(m,j)}$ is the Kronecker delta. Under Assumption \ref{A5} (a) (below on p. 13), we will denote by $\{Y(t,\tilde{\beta})\}$ the associated right-continuous Markov chain with values in $\mathbb{N}_0\times\mathbb{N}_0$ and  for each ${\beta}\in \mathscr{U}_{SM}$, the regular $Q$ process simply denoted as $p(s,(n,i),t,(m,j),{\beta})$, see [Guo and Hern$\acute{\rm a}$ndez-Lerma (2009), p. 12].\\
	Also, for each initial state $(n,i)\in \mathbb{N}_0\times\mathbb{N}_0$ at time $s=0$, we denote our probability space as $(\Omega,\mathscr{F}, P^{{\beta}}_{(n,i)})$, where $\mathscr{F}$ is Borel $\sigma$-algebra over $\Omega$ and $P^{{\beta}}_{(n,i)}$ denotes the probability measure determined by $p(s,(n,i),t,(m,j),{\beta})$. Denote $E^{{\beta}}_{(n,i)}$ as the corresponding expectation operator.
	For any real-valued measurable function $u$ on $K$ and ${\beta}\in \mathcal{U}_{SM}$, let
	\begin{equation}\label{E3.6}
	u(n,i,{\beta}):=u(n,i,{\beta}(n,i))~\forall~(n,i)\in \mathbb{N}_0\times\mathbb{N}_0~\text{and}~t\geq 0,
	\end{equation}
	whenever the integral is well defined. For any measurable function $V\geq 1$ on $\mathbb{N}_0\times\mathbb{N}_0$, we define the $V$-weighted supremum norm $\|\cdot\|$ of a real-valued measurable function $u$ on $\mathbb{N}_0\times\mathbb{N}_0$ by $$\|u\|_V:=\sup_{(n,i) \in  \mathbb{N}_0\times  S_n}\biggl\{\frac{|u(n,i)|}{V(n,i)}\biggr\},$$ and the Banach space $B_V(\mathbb{N}_0\times\mathbb{N}_0):=\{u:\|u\|_V<\infty\}$.

	Now we briefly describe the problems we consider in this paper.
	\subsection{Discounted Cost Problem}
	For ${\beta} \in \mathscr{U}_{SM}$, define $\alpha$-discounted cost criterion by 
	\begin{equation}\label{main}
	I_{\alpha}^{{\beta}} (n,i)\ = \    E_{(n,i)}^{{\beta}} \left[  
	\int_{0}^{\infty}e^{-\alpha t} r(Y(t),{\beta}(Y(t-))) dt   \right]
	\end{equation} 
	where $\alpha>0$ is the discount factor, $Y(\cdot)$ is the Markov chain corresponding to ${\beta} \in  \mathscr{U}_{SM}$ with $Y(0)=(n,i)$, $E_{(n,i)}^{ {{\beta}}}$ denote the corresponding expectation and $r $ is defined as in (\ref{cost}). 
	Here the controller wants to minimize his cost over $\mathscr{U}_{SM}$.\\
	{\bf Definition:} A control ${\beta}^*\in \mathscr{U}_{SM}$ is said to be optimal if 
	$$I_\alpha^{*}(n,i):=I_{\alpha}^{{\beta}^*} ( n,i)=\inf_{{{\beta}} \in \mathscr{U}_{SM}}I_{\alpha}^{{{\beta}}} (n,i). $$
	\subsection{Ergodic Cost Criterion}
	For ${\beta} \in \mathscr{U}_{SM}$,	the ergodic cost criterion is defined by
	\begin{equation}\label{ergodicriskcost}
	J(n,i, {\beta}) \ = \ \limsup_{T \to \infty} \frac{1}{T}  E_{(n,i)}^{{\beta}} \Big[ \int^T_0 r(Y(t), {\beta}(Y(t-))) dt 
	\Big] \, ,
	\end{equation}
	where $r $ is defined as in (\ref{cost}) and   
	$Y(\cdot)$ is the process  corresponding to the control 
	${{\beta}} \in {\mathscr{U}_{SM}}$ and $E_{(n,i)}^{ {{\beta}}}$ denote the expectation where control ${\beta}$ used with $Y(0)=(n,i)$. Here the controller wants to minimize his cost over $\mathscr{U}_{SM}$.\\
	{\bf Definition:} A control ${\beta}^*\in \mathscr{U}_{SM}$ is said to be optimal if 
	$$J^*(n,i):=J(n,i, {\beta}^{*})=\inf_{{\beta} \in \mathscr{U}_{SM}}	J(n,i, {\beta}).$$
	\subsection{Pathwise average cost criterion}
	Pathwise average cost (PAC) criterion $J_c(\cdot,\cdot,\cdot)$ is defined as follows: for all ${\beta}\in \mathscr{U}_{SM}$ and $(n,i)\in \mathbb{N}_0\times \mathbb{N}_0$,
	\begin{align}\label{EQ9.1}
	J_c(n,i,{\beta}):=\limsup_{T\rightarrow\infty}\frac{1}{T}\int_{0}^{T}r(Y(t),{\beta}(Y(t-)))dt.
	\end{align}
	{\bf Definition:} For a given $\varepsilon\geq 0$, a policy ${\beta}^{*}\in  \mathscr{U}_{SM}$ is said to $\varepsilon$-PAC-optimal if there exists a constant $g^{*}$ such that
	$$P^{{\beta}^{*}}_{(n,i)}(J_c(n,i,{\beta}^{*})\leq g^{*}+\varepsilon)=1~\text{and}~P^{{\beta}}_{(n,i)}(J_c(n,i,{\beta})\geq g^{*})=1,$$
	for all $(n,i)\in \mathbb{N}_0\times \mathbb{N}_0$ and ${\beta}\in \mathscr{U}_{SM}$. For $\varepsilon=0$, a $0$-PAC-optimal policy simply called a PAC-optimal policy.\\
	
	To ensure the regularity of a $Q$-process and finiteness of the cost criterions (\ref{main}), (\ref{ergodicriskcost}) and (\ref{EQ9.1}), we take the following assumption.
	\begin{assumption}\label{A5}
		\begin{itemize}
			\item [(a)] 	There exist a nondecreasing function $W\geq 1$ on $\mathbb{N}_0 \times \mathbb{N}_0$ and constants $c_1>0$ and $b_1\geq 0$ such that for any $ \ (n,i) \in \mathbb{N}_0 \times \mathbb{N}_0, ~\text{and}~\ \tilde{\beta}\in [\gamma,R]$, the following holds: $$\; \;\Pi_{(n,i)}^{\tilde{\beta}} W(n,i) =\sum_{(m,j)\in\mathbb{N}_0 \times \mathbb{N}_0}\pi^{\tilde{\beta}}_{(n,i)(m,j)}W(m,j)\leq -c_1 W(n,i) + b_1\delta_{(n,i)(0,0)} , \,  $$ where $\delta_{(n,i)(m,j)}$ is the Dirac delta measure.\\
			\item [(b)] For every $(n,i)\in \mathbb{N}_0 \times \mathbb{N}_0 $ and some constant $M >0$, $r(n,i,\tilde{\beta}) \leq  M W(n,i)$.

		\end{itemize}
	\end{assumption}
\begin{remark}
	\begin{itemize}
		\item [(1)]  Assumption \ref{A5} (a) and its variants are used to study ergodic control problem, see, Guo and Hern$\acute{\rm a}$ndez-Lerma (2009); Meyn and Tweedie (1993); Pal and Pradhan (2019). 
		\item [(2)]
		Assumption \ref{A5} (b) and its variants are very useful Assumption for unbounded costs in control theory, see Golui and Pal (2022); Guo and Hern$\acute{\rm a}$ndez-Lerma (2009).	For bounded cost as in [Ghosh and Saha (2014); Kumar and Pal (2015)], Assumption \ref{A5} (b) is not required. By (\ref{eq 3.1}) and (\ref{cost}), we have that the functions, $r(n,i,{\tilde{\beta}}),\; \pi_{(n,i)(m,j)}^{{\tilde{\beta}}}$, and 
		$\sum_{(m,j)\in\mathbb{N}_0 \times \mathbb{N}_0}\pi^{\tilde{\beta}}_{(n,i)(m,j)}W(m,j)$ are all continuous in $\tilde{\beta}$ for each fixed $(n,i)\in \mathbb{N}_0 \times \mathbb{N}_0,$ with $W$ as in Assumption \ref{A5}. To ensure the existence of optimal stationary strategies, we need this continuity (see, for instance, [Ghosh and Saha (2014); Kumar and Pal (2013), (2015)] and their references).
	\end{itemize}
	
\end{remark}
Now to prove the existence of an optimal stationary policy for discounted cost criterion, we need the following Assumption, see [Guo and Hern$\acute{\rm a}$ndez-Lerma (2009), chapter 6].
	\begin{assumption}\label{A2}
		There exists a nonnegative function $W'$ on $ \mathbb{N}_0 \times \mathbb{N}_0$ and constants $c'>0, \; b'\geq 0$, and $ M' > 0$ such that
		\begin{align*}
		&	\pi_{(n,i)}^* W(n,i) \leq M' W'(n,i) , \ (n,i) \in \mathbb{N}_0 \times \mathbb{N}_0, \\
		&\sum_{(m,j)\in\mathbb{N}_0 \times \mathbb{N}_0}\pi^{\tilde{\beta}}_{(n,i)(m,j)}W'(m,j) \leq  c' W'(n,i) + b' , \ (n,i) \in \mathbb{N}_0 \times \mathbb{N}_0, \ \tilde{\beta} \in [\gamma,R]. \, 
		\end{align*}
	\end{assumption}
	We now state an important condition that is satisfied by our transition rates given by \eqref{eq 3.1}.\\
	\textbf{Condition A:}
	For each $\beta\in \mathscr{U}_{SM}$, the corresponding Markov process $\{Y(t)\}$ with transition function $p((n,i),t,(m,j),\beta)$ is irreducible, which means that, for any two states $(n,i)\neq (m,j)$, there exists a set of distinct states $(n,i)=(m_1,i_1),\cdots,(m_k,i_k)$ such that $$\pi_{(m_1,i_1)(m_2,i_2)}^{\beta(m_1,i_1)}\cdots\pi_{(m_k,i_k)(m,j)}^{\beta(m_k,i_k)}>0.$$

	\begin{remark}
		\begin{itemize}
			\item [(1)] Condition A is satisfied by our transition rates given by \eqref{eq 3.1}.
			\item [(2)]Under Assumptions \ref{A5} and Condition A, for each $\beta\in \mathscr{U}_{SM}$, by [Guo and Hern$\acute{\rm a}$ndez-Lerma (2009), Propositions C.11 and C.12], we say that the MC $\{Y(t)\}$ has a unique invariant probability measure, $\vartheta_\beta$ which satisfies $$\vartheta_\beta(m,j)=\lim_{t\rightarrow\infty}p((n,i),t,(m,j),\beta)~(\text{ independent of}~(n,i)\in \mathbb{N}_0\times\mathbb{N}_0) ~\text{for all}~ (m,j)\in \mathbb{N}_0\times\mathbb{N}_0.$$ Thus by Assumption \ref{A5} (a) and [Guo and Hern$\acute{\rm a}$ndez-Lerma (2009), Lemma 6.3 (i)], we have $$\vartheta_\beta(W):=\sum_{(m,j)}W(m,j)\vartheta_\beta(m,j)\leq \frac{b_1}{c_1}<\infty,$$ and so, 
			\begin{equation}\label{EQN8.1}
			\vartheta_\beta(u):=\sum_{(m,j)}u(m,j)\vartheta_\beta(m,j)<\infty,~\forall \beta\in \mathscr{U}_{SM}~\text{for any}~ u\in B_W(\mathbb{N}_0 \times \mathbb{N}_0).
			\end{equation}
		
		\end{itemize}
		
	\end{remark}
	To get the existence of average cost optimal (ACO) stationary strategy, in addition to Assumptions \ref{A5} and \ref{A2}, we impose the following condition. This assumption is very important to study a ergodic control problem, see [Guo and Hern$\acute{\rm a}$ndez-Lerma (2009), chpter 7]. Under this assumption, the Markov chain is uniformly ergodic.
\begin{assumption}\label{A3}
The control model is uniformly ergodic, which means the following: there exist constants $\delta>0$ and $L_2>0$ such that (using the notation in (\ref{EQN8.1})) $$\sup_{\beta\in \mathscr{U}_{SM}}|E^{\beta}_{(n,i)}u(Y(t))-\vartheta_\beta(u)|\leq L_2 e^{-\delta t}\|u\|_W W(n,i)$$ for all $(n,i)\in \mathbb{N}_0\times\mathbb{N}_0$, $u\in B_W(\mathbb{N}_0 \times \mathbb{N}_0)$, and $t\geq 0$.
\end{assumption}
	To get the existence of pathwise average cost optimal (PACO) stationary strategy, in addition to Assumptions \ref{A5}, \ref{A2} and \ref{A3}, we impose the following conditions. 
	\begin{assumption}\label{A4}
		Let $W\geq 1$ be as in Assumption \ref{A5}. For $k=1,2$, there exist nonnegative functions $W^{*}_k\geq 1$ on $\mathbb{N}_0$ and constants $c^{*}_k>0$, $b^{*}_k\geq 0$, and $M^*_k>0$ such that for all $(n,i)\in \mathbb{N}_0\times\mathbb{N}_0$ and $\tilde{\beta}\in [\gamma, R]$,
		\begin{itemize}
			\item [(a)] $W^2(n,i)\leq M^*_1W_1^*(n,i)$ and $\displaystyle\sum_{(m,j)}\pi^{\tilde{\beta}}_{(n,i)(m,j)}W^*_1(n,j)\leq -c_1^*W^*_1(n,i)+b_1^{*}$.\\
			\item [(b)] $[\pi_{(n,i)}^* W(n,i)]^2\leq M^*_2W_2^*(n,i)$ and $\sum_{(m,j)}\pi^{\tilde{\beta}}_{(n,i)(m,j)}W^*_2(m,j)\leq -c_2^*W^*_2(n,i)+b^*_2.$
		\end{itemize}
	\end{assumption}
	\section{Analysis of Discounted Cost Problem}
	In this section we study the infinite horizon discounted cost problem given by the criterion (\ref{main}) and prove the existence of optimal policy. Corresponding to the cost criterion (\ref{main}), we recall the following function 
	\begin{equation*}
	I^*_{\alpha}(n,i)= \inf_{{\beta} \in \mathscr{U}_{SM}}I_{\alpha}^{{\beta}} (n,i).
	\end{equation*}
	Using the dynamic programming heuristics, the Hamilton-Jacobi-Bellman (HJB) equations for discounted cost criterion are given by
	\begin{eqnarray}\label{discount_hjb}
	\alpha I^*_{\alpha}(n,i)&=& \inf_{\tilde{\beta}\in [\gamma,R]} \Big [ \Pi_{(n,i)}^{\tilde{\beta}} 
	I^*_{\alpha}(n,i) + r(n,i,\tilde{\beta}) \Big ]\ , 
	\end{eqnarray}
	where $\Pi_{(n,i)}^{\tilde{\beta}} f(n,i):= \displaystyle{\sum_{(m,j)\in \mathbb{N}_0 \times \mathbb{N}_0}} \pi_{(n,i)(m,j)}^{\tilde{\beta}}f(m,j)$, for any function $f(n,i)$.\\
	Define an operator $T:B_W(\mathbb{N}_0 \times \mathbb{N}_0)\rightarrow B_W(\mathbb{N}_0 \times \mathbb{N}_0)$ as
	\begin{equation} \label{E5.7}
	Tu(n,i):= \mathop{\inf}\limits_{\tilde{\beta}\in [\gamma,R]} \biggl [\dfrac{r(n,i,\tilde{\beta})}{R+\alpha+\lambda+\mu}+\dfrac{R+\lambda+\mu}{R+\alpha+\lambda+\mu} \mathop{\sum}\limits_{(m,j)\in \mathbb{N}_0 \times \mathbb{N}_0}p^{\tilde{\beta}}_{(n,i)(m,j)}u(m,j) \biggr ],
	\end{equation} 	
	for $u\in B_W(\mathbb{N}_0 \times \mathbb{N}_0)$ and $(n,i)\in \mathbb{N}_0\times\mathbb{N}_0$, where $$p^{\tilde{\beta}}_{(n,i)(m,j)}:=\frac{\pi^{\tilde{\beta}}_{(n,i)(m,j)}}{R+\lambda+\mu}+\delta_{(n,i)(m,j)}$$ is a probability measure on $\mathbb{N}_0\times\mathbb{N}_0$ for each $(n,i,\tilde{\beta})\in\mathbb{N}_0\times\mathbb{N}_0\times [\gamma, R]$ and $\delta_{(n,i)(m,j)}$ is the Dirac-delta function.\\
	Next we prove the optimality theorem for the discounted cost criterion. In this theorem, we find the existence of solution of discounted-cost optimality equation (DCOE) and optimal stationary policy.
	\begin{thm}\label{T6}
		Suppose that Assumptions \ref{A5} and \ref{A2} hold.
		Define $u_0:=0$, $u_{k+1}:=Tu_k$.
		Then the following hold.
		\begin{itemize}
			\item [(a)] The sequence $\{u_k\}_{k\geq 0}$ is monotone nondecreasing, and the limit $u^*:=\lim_{k\rightarrow\infty}u_k$ is in $B_W(\mathbb{N}_0 \times \mathbb{N}_0)$.\\
			\item [(b)] The function $u^*$ in (a) satisfies the fixed-point equation $u^*=Tu^*$, or, equivalently, $u^*$ verifies the DCOE, that is
			\begin{equation}\label{EQ5.7}
			\alpha u^*(n,i)=\inf_{\tilde{\beta} \in [\gamma,R]}\biggl\{r(i,n,\tilde{\beta})+\mathop{\sum}\limits_{(m,j)\in \mathbb{N}_0 \times \mathbb{N}_0}\pi^{\tilde{\beta}}_{(n,i)(m,j)}u^*(m,j)\biggr\}~\forall (n,i)\in \mathbb{N}_0\times\mathbb{N}_0.
			\end{equation}
			
			\item [(c)] There exist stationary policies $\beta_k$ (for each $k\geq 0$) and $\beta^*_{\alpha}$ attaining the minimum in the equations $u_{k+1}=Tu_k$ and the DCOE (\ref{EQ5.7}), respectively.  Moreover, $u^*=I^*_{\alpha}$. and the policy $\beta^*_\alpha$ is discounted-cost optimal.\\
			\item [(d)] Every limit point in $\mathscr{U}_{SM}$ of the sequence $\{\beta_k\}$ in (c) is a discounted-cost optimal stationary policy.
		\end{itemize}
	\end{thm}
	\begin{proof}
		\begin{itemize}
			\item [(a)] We first prove the monotonicity of $\{u_k\}_{k\geq 0}$. Let $u_0=0$. Since $r\geq 0$, $u_1(n,i)\geq u_0(n,i)$ for all $(n,i)\in \mathbb{N}_0 \times \mathbb{N}_0$. Consequently, the monotonicity of $T$ gives
			$$u_k=T^ku_0\leq T^ku_1=u_{k+1}~\text{ for every}~k\geq 1.$$ So, the sequence $\{u_k\}_{k\geq 0}$ is a monotone increasing sequence. So, the limit $u^*$ exists. Also, by direct calculations we get
			$$|u_k(n,i)|\leq \frac{b_1M}{\alpha(\alpha+c_1)}+\frac{MW(n,i)}{\alpha+c_1}\leq \frac{(\alpha+b_1)M}{\alpha(\alpha+c_1)}W(n,i)~\forall k\geq 0~\text{and}~(n,i)\in \mathbb{N}_0 \times \mathbb{N}_0,$$ which implies that $\sup_{k\geq 0}\|u_k\|_W$ is finite. Hence $u^*\in B_W(\mathbb{N}_0 \times \mathbb{N}_0)$.
			\item [(b)] By the monotonicity of $T$, $Tu^*\geq Tu_k=u_{k+1}$ for all $k\geq 0$, and thus 
			\begin{align}\label{EQ7.19}
			Tu^*\geq u^*.
			\end{align}
			Now, there exists $\beta_k\in \mathscr{U}_{SM}$ such that
			\begin{equation*} 
			u_{k+1}(n,i)= \biggl [\dfrac{r(n,i,\beta_k)}{R+\alpha+\lambda+\mu}+\dfrac{R+\lambda+\mu}{R+\alpha+\lambda+\mu} \mathop{\sum}\limits_{(m,j)}p^{\beta_k}_{(n,i)(m,j)}u_k(m,j)  \biggr ],
			\end{equation*} 
			for all $k\geq 0$. Since $\mathscr{U}_{SM}$ is compact, there exist a policy $\beta^*\in \mathscr{U}_{SM}$ and a subsequence of $k$ for which $\lim_{k\rightarrow\infty}\beta_k=\beta^*$.
			So, by the generalized Fatou's lemma by taking $k\rightarrow\infty$, we get
			\begin{equation*} 
			u^*(n,i)\geq \biggl [\dfrac{r(n,i,\beta^*)}{R+\alpha+\lambda+\mu}+\dfrac{R+\lambda+\mu}{R+\alpha+\lambda+\mu} \mathop{\sum}\limits_{(m,j)}p^{\beta^*}_{(n,i)(m,j)}u^*(m,j)  \biggr ]~\forall (n,i)\in \mathbb{N}_0 \times \mathbb{N}_0,
			\end{equation*}
			which gives $u^*\geq Tu^*$. So, $u^*= Tu^*$, and so we get DCOE (\ref{EQ5.7}). 
			\item [(c)] Since we have that $u_k$ and $u^*$ are in $B_W(\mathbb{N}_0 \times \mathbb{N}_0)$, from [Guo and Hern$\acute{\rm a}$ndez-Lerma (2009), Proposition A.4], we see that the functions in (\ref{E5.7}) and (\ref{EQ5.7}) are continuous in $\tilde{\beta}\in [\gamma,R]$. Hence the first claim of part (c) holds.
			Moreover, for all $\forall (n,i)\in \mathbb{N}_0 \times \mathbb{N}_0$ and $\beta\in \mathscr{U}_{SM}$, it follows from (\ref{EQ5.7}) that
			\begin{equation}\label{EQ7.20}
			\alpha u^*(n,i)\leq \biggl\{r(n,i,\beta)+\Pi^{\beta}_{(n,i)}u^*(n,i)\biggr\}~\text{ for all}~ (n,i)\in \mathbb{N}_0 \times \mathbb{N}_0
			\end{equation}
			with equality if $\beta=\beta^*_\alpha$. Hence, (\ref{EQ7.20}), together with [Guo and Hern$\acute{\rm a}$ndez-Lerma (2009), Theorem 6.9 (b)], yeilds that $$I_\alpha^{\beta^*_\alpha}(n,i)=u^*(i)\leq I_\alpha^{\beta}(n,i)~\text{ for all}~ (n,i)\in \mathbb{N}_0 \times \mathbb{N}_0~\text{and}~t\geq 0.$$ Hence, we prove part (c).
			\item [(d)] By part (a) and the generalized dominated convergence theorem in [Guo and Hern$\acute{\rm a}$ndez-Lerma (2009), Proposition A.4], every limit point $\beta\in \mathscr{U}_{SM}$ of $\{\beta_k\}$ satisfies
			\begin{equation*} 
			u^*(n,i) =  \biggl [\dfrac{r(n,i,\beta)}{R+\alpha+\lambda+\mu}+\dfrac{R+\lambda+\mu}{R+\alpha+\lambda+\mu} \mathop{\sum}\limits_{(m,j)}p^\beta_{(n,i)(m,j)}u^*{(m,j)}  \biggr ],
			\end{equation*}
			which is equivalent to \begin{equation*}
			\alpha u^*(n,i)= \biggl\{r(n,i,\beta)+\Pi^{\beta}_{(n,i)}u^*(n,i)\biggr\}~\forall  (n,i)\in \mathbb{N}_0 \times \mathbb{N}_0.
			\end{equation*}
			Thus by (b) and [Guo and Hern$\acute{\rm a}$ndez-Lerma (2009), Theorem 6.9 (c)], $I^\beta_\alpha(n,i)=u^*(n,i)=I^*_\alpha(n,i)$ for every $ (n,i)\in \mathbb{N}_0 \times \mathbb{N}_0$.
		\end{itemize}
	\end{proof}
	\subsection{\textbf{The Discounted-Cost Value Iteration Algorithm:}}
	Now using value iteration algorithm, we find an optimal production rate $   \beta^{*}_\alpha$ for discounted-cost criterion.
	Since this optimal production rate $\beta^{*}_\alpha$ cannot be computed explicitly, we explore the possibility of algorithmic computation. Thus, in the presence of Theorem \ref{T6}, one can use the following value iteration algorithm for computing $\beta^{*}_\alpha$.\\
	\textbf{\textit{A Value Iteration Algorithm 4.1:}}\label{AL3} By the value iteration algorithm, we will find an optimal production rate $\beta^{*}_\alpha$, described briefly as follows:
	~\\{\bf Step 0:}
	Let $v_0 {(n,i)}= {0}$, \ for all \ $(n,i) \in \mathbb{N}_0\times \mathbb{N}_0.$
	~\\ 
	{\bf Step 1:} For $k\ge 1$, define
	\begin{equation} \label{value}
	v_k{(n,i)} = \mathop{\inf}\limits_{\tilde{\beta} \in [\gamma,R]} \biggl [\dfrac{r(n,i,\tilde{\beta})}{R+\alpha+\lambda+\mu}+\dfrac{R+\lambda+\mu}{R+\alpha+\lambda+\mu} \mathop{\sum}\limits_{(m,j)}p^{\tilde{\beta}}_{(n,i)(m,j)}v_{k-1}{(m,j)}\biggr ],
	\end{equation}  where $(n,i),(m,j)\in \mathbb{N}_0\times \mathbb{N}_0$, $p^{\tilde{\beta}}_{(n,i)(m,j)}:=\frac{\pi^{\tilde{\beta}}_{(n,i)(m,j)}}{R+\lambda+\mu}+\delta_{(n,i)(m,j)}$.
	~\\
	{\bf Step 2:} Choose $\beta_k\in \mathscr{U}_{SM}$ attaining the minimum in the right-hand side of (\ref{value}). 
	~\\
	{\bf Step 3:} $v_{*}{(n,i)}=\mathop{\lim}\limits_{k\rightarrow \infty}v_k{(n,i)},$ \ for all \ $(n,i) \in \mathbb{N}_0\times \mathbb{N}_0.$
	~\\
	{\bf Step 4:} Every limit point in $\mathscr{U}_{SM}$ of the sequence $\lbrace \beta_k \rbrace$ is a discounted-cost optimal stationary policy.
	
	\subsection*{Numerical Example:} Now, we discuss the results obtained from the implementation of discounted-cost value iteration algorithm. Unless stated otherwise, the parameters are considered as $\lambda=3, \mu=5, \alpha=0.7, h=\$100, c_1=\$20, c_2=\$30, c_3=\$40, S=2$ and $[\gamma,R] = [0.001, 2]$ discretized as $\{0.001, 0.002, 0.003, \dots, 1.999, 2\}$ for computational purposes. Note that $n$ and $i$ are assumed to range from $0$ to $4$. Figure \ref{figure1} shows the speed of convergence of the value function for selected states and different $\alpha$ values. 
	
	\begin{figure}%
		\centering
		\subfloat [\centering Convergence of value iteration in a finite number of steps.]
		{
			{\includegraphics[width=7.5cm]{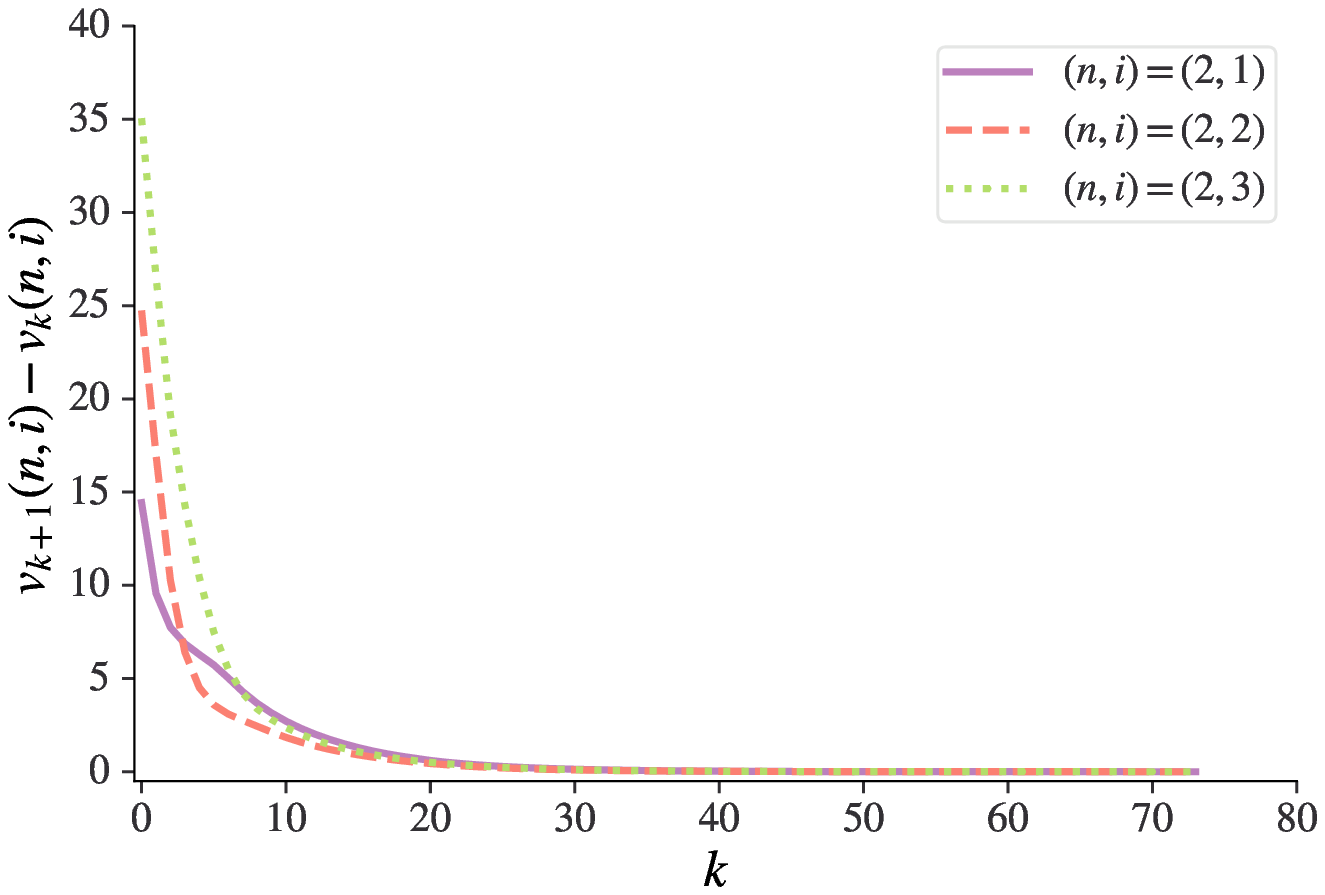}}
		}
	\subfloat [\centering Impact of $\alpha$ on the convergence speed.]
		{ 
			{\includegraphics[width=7.5cm]{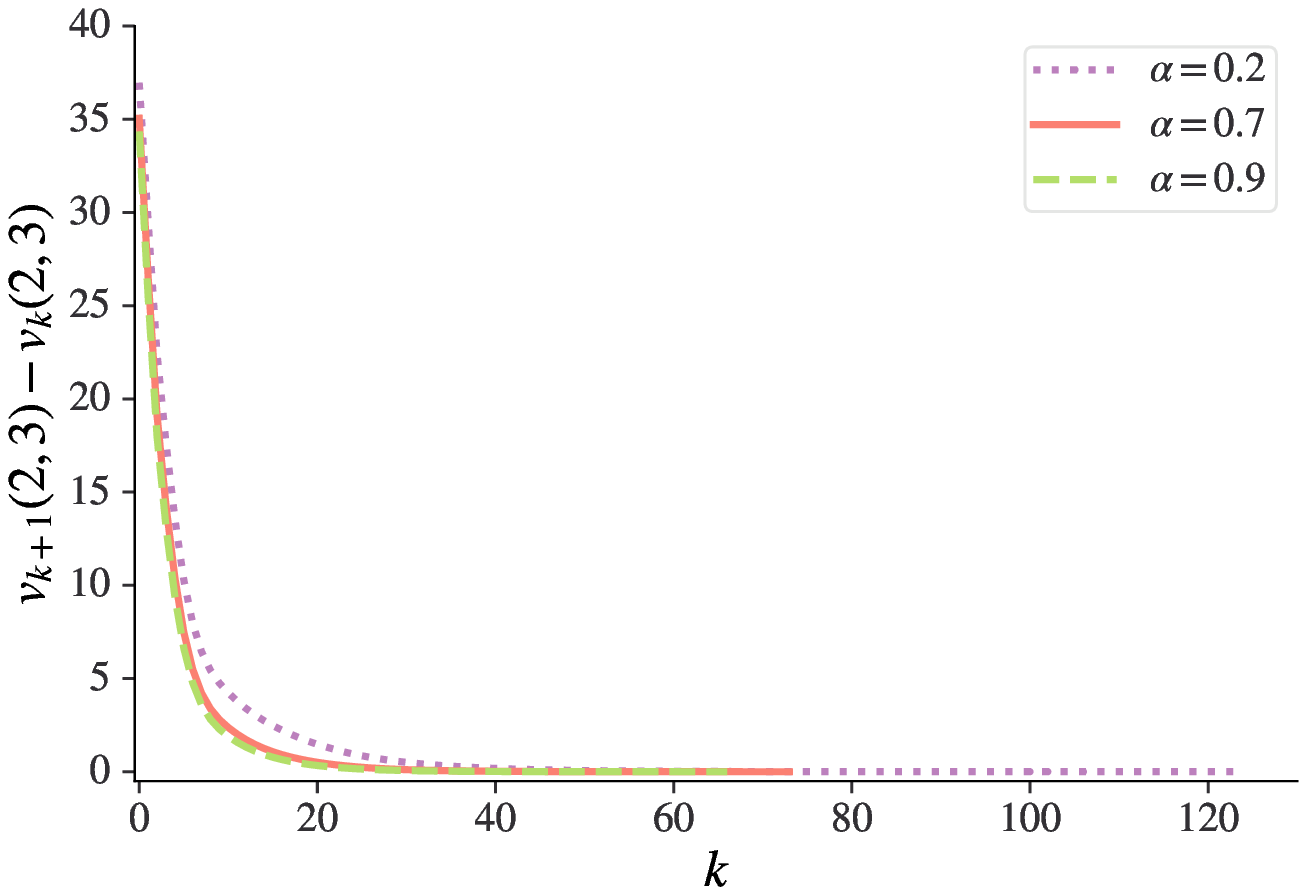}}%
		}
		\caption{$\epsilon-$convergence of $v_k$ for selected state(s) as $k$ increases, where $\epsilon = 0.001$.}
		\label{figure1}%
	\end{figure}
	
	\begin{figure}[H]
		\centering
		\includegraphics[width=0.6\textwidth]{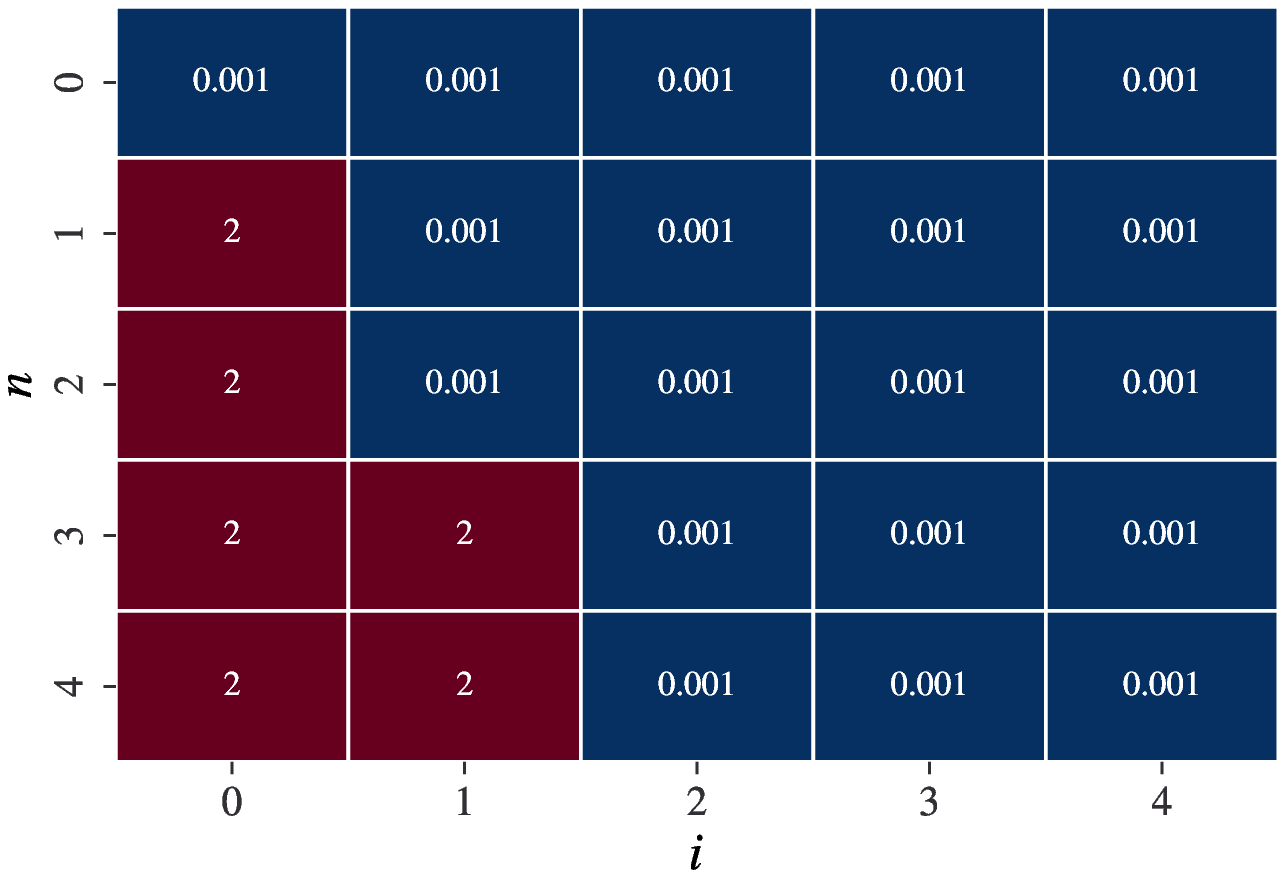}
		\caption{Optimal $\beta^*_\alpha$ for each state $(n,i)$ obtained from the discounted cost value iteration.}
		\label{figure2}
	\end{figure}
	
	The optimal policy table for the discounted-cost criterion using the input values above is shown in Figure \ref{figure2}. A lower production rate is advised in most states where the majority of customer demands can be fulfilled using the existing inventory. We notice that a high production rate is optimal for some states where there is zero/low inventory. As $n$ increases, there is a need to produce more items per unit time, as reflected in the optimal policy for such states. We also see that the optimal policy is in accordance with the current state as well as future transitions. It is important to note that the optimal policy has a strong correlation with the service rate of the system $\mu$. As $\mu$ increases, the expected service time reduces and thus the production frequency should be increased in consideration of future demands. This effect of $\mu$ is evident in the optimal policy tables given in Figure \ref{figure3}.
	
	\begin{figure}%
		\centering
	\subfloat [\centering $\mu=3$] 
		{
			{\includegraphics[width=7.5cm]{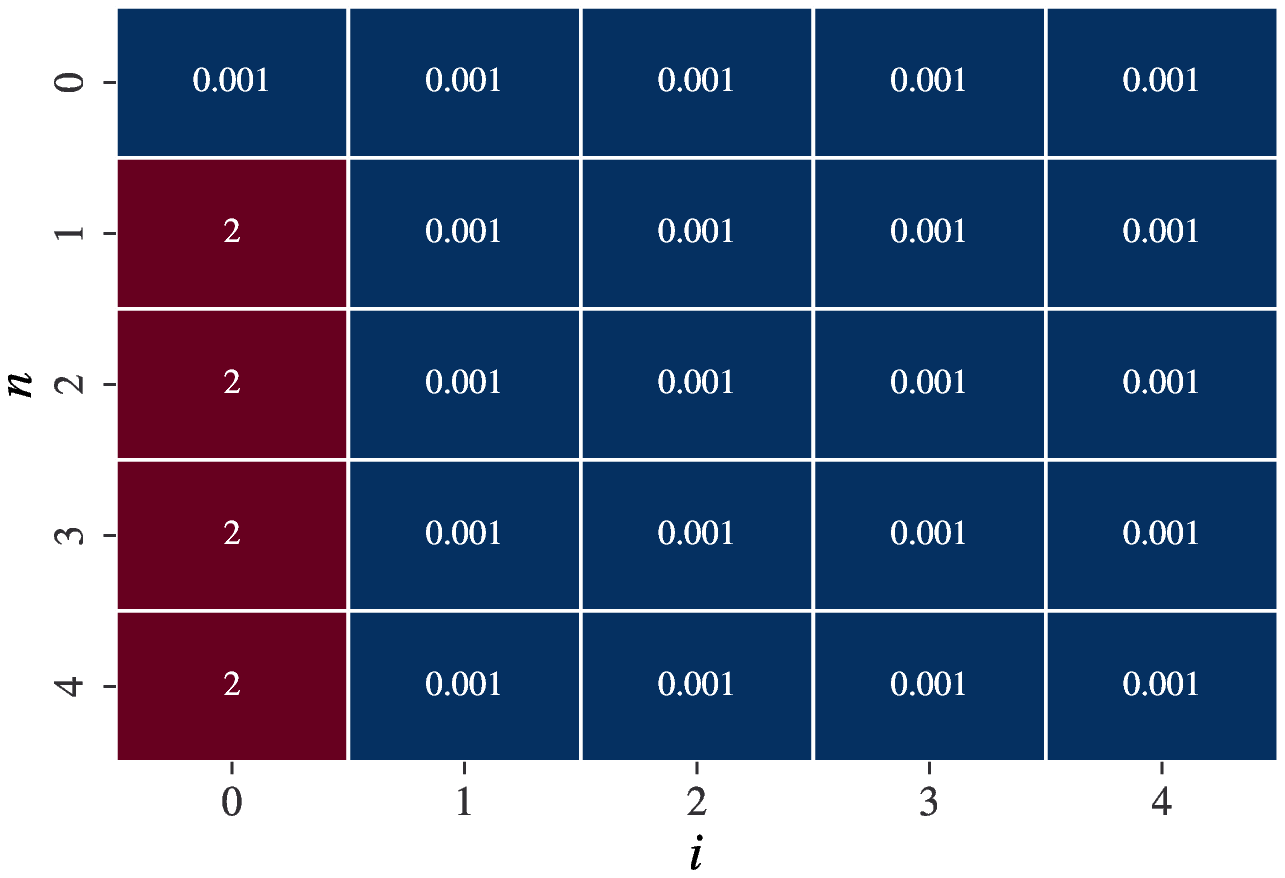}}%
		}
	\subfloat [\centering $\mu=10$] 
		{
			{\includegraphics[width=7.5cm]{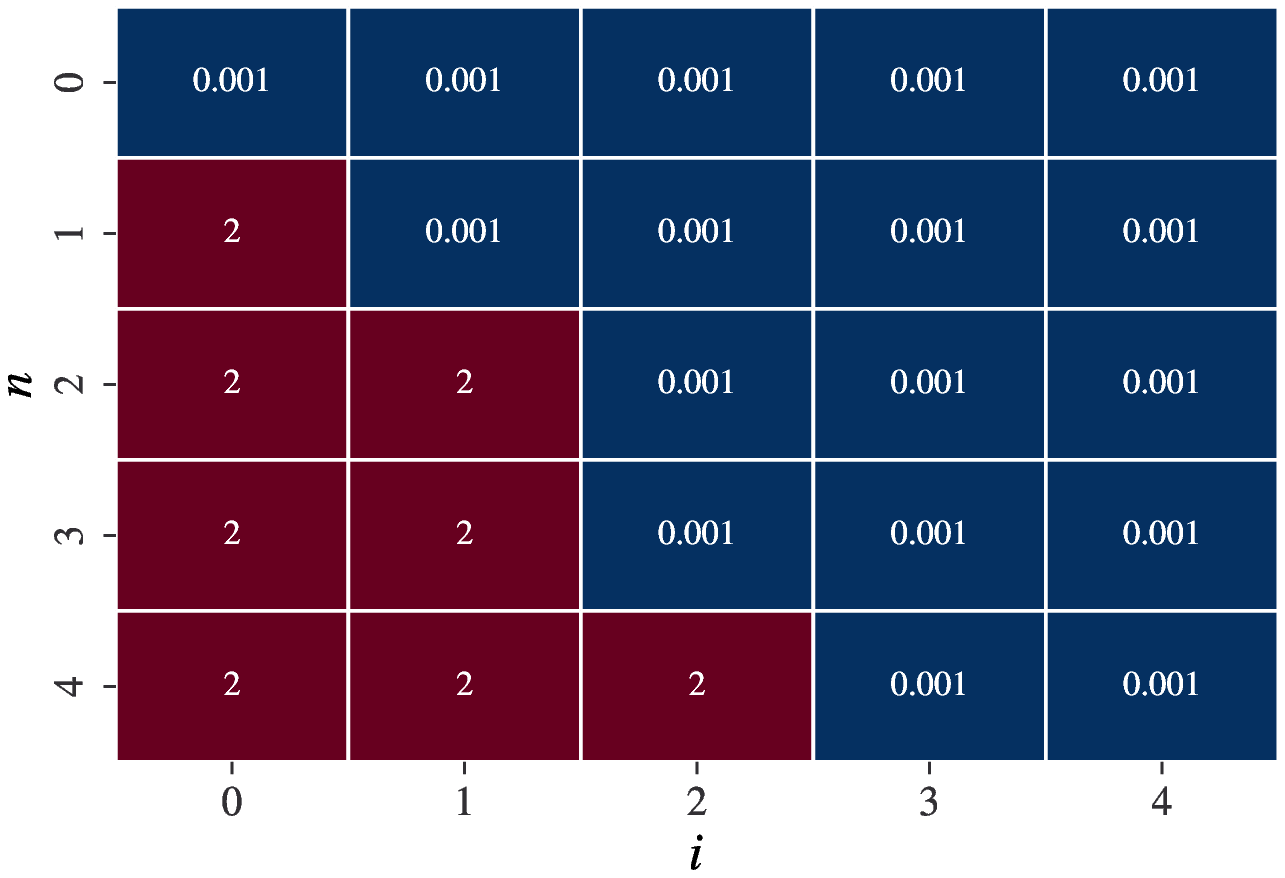}}%
		}
		\caption{Optimal policy tables for varying $\mu$.}
		\label{figure3}%
	\end{figure}

	\subsection{\textbf{The Discounted-Cost Policy Iteration Algorithm:}}Now if the state and action spaces are both finite then using Lemma \ref{L6} and Theorem \ref{T4} below, one can find an optimal production rate $\beta_\alpha^{*}$ by using the policy iteration algorithm given below.\\
	In order to solve the discounted-cost problem through the policy iteration algorithm, we define some sets. For every $\beta\in \mathscr{U}_{SM}$, $(n,i)\in \mathbb{N}_0\times\mathbb{N}_0$, and $\tilde{\beta}\in [\gamma,R]$, let
	\begin{equation}\label{EQ7.13}
	D_{\beta}(n,i,\tilde{\beta}):=r(n,i,\tilde{\beta})+\sum_{(m,j)\in\mathbb{N}_0 \times \mathbb{N}_0}\pi^{\tilde{\beta}}_{(n,i)(m,j)}I^{\beta}_{\alpha}(m,j)
	\end{equation}
	and
	\begin{equation}\label{E5.3}
	E_\beta(n,i):=\{\tilde{\beta}\in [\gamma,R]:D_\beta(n,i,\tilde{\beta})<\alpha I^{\beta}_{\alpha}(n,i) \}.
	\end{equation}
	We then define an improvement policy $\beta^{'}\in \mathscr{U}_{SM}$ (depending on $\beta$) as follows:\\
	\begin{equation}\label{EQ7.15}
	\beta^{'}(n,i)\in E_\beta(n,i)~\text{if}~E_\beta(n,i)\neq \emptyset~\text{and}~\beta^{'}(n,i):=\beta(n,i)~\text{if}~E_\beta(n,i)=\emptyset.
	\end{equation}
	\textbf{Note:}
	Now if the number of customers is $m$ and the number of items is also $m$, then corresponding to fixed $\beta\in \mathscr{U}_{SM}$, let $I$ is the ${m^2\times m^2}$ standard identity matrix. Also, define $\hat{I}^{\beta}_{\alpha}:=[{I}^{\beta}_{\alpha}(n,i)]_{m^2\times 1}$ and $\hat{r}(\beta):=[r(n,i,\beta)]_{m^2\times 1}$ are column vectors (here $\beta$ is fixed but $(n,i)$ will vary).\\
	Next we state a Lemma whose proof is in [Guo and Hern$\acute{\rm a}$ndez-Lerma (2009), Lemma 4.16, Lemma 4.17].
	\begin{lemma}\label{L6}
		Suppose that Assumption \ref{A5} holds.
		Then for the finite CTMDP model, $I^\beta_{\alpha}$ is a unique bounded solution to the equation
		\begin{equation}\label{EQ7.12}
		\alpha u(n,i)=r(n,i,\beta)+\sum_{(m,j)\in\mathbb{N}_0 \times \mathbb{N}_0}\pi^{{\beta}}_{(n,i)(m,j)}u(m,j)~\forall (n,i)\in \mathbb{N}_0\times\mathbb{N}_0,~\text{ for every}~\beta\in \mathscr{U}_{SM}.
		\end{equation}
		Also, for any given $\beta\in \mathscr{U}_{SM}$, let $\beta^{'}\in \mathscr{U}_{SM}$ be defined as in (\ref{EQ7.15}) and suppose that $\beta^{'}\neq \beta$. Then $I_\alpha^\beta\geq I_\alpha^{\beta^{'}}$.
	\end{lemma}
\noindent	\textbf{The Policy Iteration Algorithm  4.1:}\label{Al2}\\
	{\bf Step 1:} Pick an arbitrary $\beta\in \mathscr{U}_{SM}$. Let $k=0$ and take $\beta_k:=\beta$\\
	{\bf Step 2:} (Policy evaluation) Obtain $\hat{I}^{\beta_k}_{\alpha}=[\alpha I-Q(\beta_k)]^{-1}\hat{r}(\beta_k)$ (by Lemma \ref{L6}), where $Q(\beta_k)=\left[\pi^{\beta_k}_{(n,i)(m,j)}\right]$ as defined on p. 11, $I$ is the identity matrix, $\hat{I}^{\beta_k}_{\alpha}$ and $\hat{r}$ are column vecors.\\
	{\bf Step 3} (Policy improvement) Obtain a policy $\beta_{k+1}$ from (\ref{EQ7.15}) (with $\beta_k$ and $\beta_{k+1}$ in lieu of $\beta$ and $\beta^{'}$, respectively.\\
	{\bf Step 4:} If $\beta_{k+1}=\beta_k$, then stop because $\beta_{k+1}$ is discounted-cost optimal (by Theorem \ref{T4} below). Otherwise, increase $k$ by 1 and return to Step 2.\\
	To get the optimal policy from the above policy iteration algorithm, we prove the following Theorem. 
	\begin{thm}\label{T4}
		Suppose that Assumption \ref{A5} holds.
		Then for each fixed discounted factor $\alpha>0$, the discounted-cost policy iteration algorithm yields a discounted-cost optimal stationary policy in a finite number of iterations.
	\end{thm}
	\begin{proof}
		Let $\{\beta_k\}$ be the sequence of polices in the discounted-cost policy iteration algorithm above. Then, by Lemma \ref{L6}, we have $I_\alpha^{\beta_k}\succeq I^{\beta_{k+1}}_\alpha$. Thus, each policy in the sequence $\{\beta_k, k=0,1,\cdots\}$ is different. Since the number of polices is finite, the iterations must stop after a finite number. Suppose that the algorithm stops at a policy denoted by $\beta^*_\alpha$. Then $\beta^*_\alpha$ satisfies the optimality equation
		\begin{align}\label{EQ7.3}
		\alpha I_\alpha^*(n,i)=\inf_{\tilde{\beta}\in [\gamma,R]}\biggl[{r(n,i,\tilde{\beta})}+ \mathop{\sum}\limits_{(m,j)}\pi^{\tilde{\beta}}_{(n,i)(m,j)}	I_\alpha^*{(m,j)}  \biggr ].
		\end{align}
		Thus, by [Guo and Hern$\acute{\rm a}$ndez-Lerma (2009), Theorem 4.10], $\beta^*_\alpha$ is discounted-cost optimal.
	\end{proof}
	
\noindent \textbf{Note:} As expected, the above algorithm with a discrete action space given by\\
$\{0.001, 0.002, 0.003, \dots, 1.999, 2\}$ provides the same optimal solution as the value iteration algorithm in Figure \ref{figure2}. Differences may appear in the optimal policy tables (corresponding to each algorithm) when there are alternate optimal production rates for one or more states. The speed of convergence of the policy iteration depends on the initial choice of arbitrary $\beta(n,i)$ for each state $(n,i)$.
	
	\normalsize

	\section{Analysis of Ergodic Cost Criterion}
	In this section  we prove that under Assumptions \ref{A5}, \ref{A2} and \ref{A3}, the average cost optimality equation (ACOE) (or HJB equation) given by \eqref{EQ8.1} has a solution. Also, we find the optimal stationary policy by using policy iteration algorithm for this cost criterion.\\

Next we prove the optimality theorem for ergodic cost criterion.
	\begin{thm}\label{T9}
		Suppose that Assumptions \ref{A5}, \ref{A2} and \ref{A3} hold. Then:
		\begin{itemize}
			\item [(a)] There exists a solution $(g^*,\tilde{u})\in \mathbb{R}\times B_W(\mathbb{N}_0 \times \mathbb{N}_0)$ to the ACOE
				\begin{align}\label{EQ8.1}
	g^*=\inf_{\tilde{\beta} \in [\gamma,R]}\biggl\{r(n,i,\tilde{\beta})+\sum_{(m,j)}\tilde{u}(m,j)\pi_{(n,i)(m,j)}^{\tilde{\beta}}\biggr\}~\forall (n,i)\in \mathbb{N}_0\times\mathbb{N}_0.
	\end{align}
	Moreover, the constant $g^*$ coincides with the optimal average cost function $J^*$, i.e., $$g^*=J^*(n,i)~\forall (n,i)\in \mathbb{N}_0\times\mathbb{N}_0,$$ and $\tilde{u}$ is unique up to additive constants.
			\item [(b)] A stationary  policy $\beta^*\in \mathscr{U}_{SM}$ is AC optimal iff it attains the minimum in ACOE (\ref{EQ8.1}) i.e.,
				\begin{align}\label{EQ8.2}
	g^*=\biggl\{r(n,i,\beta^*)+\sum_{(m,j)}\tilde{u}(m,j)\pi_{(n,i)(m,j)}^{\beta^*}\biggr\}~\forall (n,i)\in \mathbb{N}_0\times\mathbb{N}_0.
	\end{align} 
		\end{itemize}
	\end{thm}
	\begin{proof}
	We prove part (a) and (b) together. Take the $\alpha$-discounted cost optimal stationary policy $\beta^*_\alpha$ as in Theorem \ref{T6}. Hence $I_\alpha^{\beta^*_\alpha}(n,i)=I_\alpha^{*}(n,i)$. Now define $u_\alpha^{\beta^*_\alpha}(n,i):=I_\alpha^{\beta^*_\alpha}(n,i)-I_\alpha^{\beta^*_\alpha}(n_0,i_0)$, where $(n_0,i_0)$ is a fixed reference state. Now we apply the vanishing discounted approach. By [Guo and Hern$\acute{\rm a}$ndez-Lerma (2009), Lemma 7.7, Proposition A.7], we get a sequence $\{\alpha_k\}$ of discounted factors such that $\alpha_k\downarrow 0$, a constant $g^*$ and a function $\bar{u}\in B_W(\mathbb{N}_0 \times \mathbb{N}_0)$ such  that
		\begin{align}\label{EQ8.4}
		\lim_{k\rightarrow\infty}\alpha_k I^*_{\alpha_k}(n_0,i_0)=g^*~\text{and}~\lim_{k\rightarrow\infty}u^{\beta^*_{\alpha_k}}_{\alpha_k}(n,i)=\bar{u}(n,i)~\forall (n,i)\in \mathbb{N}_0 \times \mathbb{N}_0.
		\end{align}
		Now for all $k\geq 1$ and $(n,i)\in \mathbb{N}_0 \times \mathbb{N}_0$, by Theorem \ref{T6}, we have
		$$\frac{\alpha_k I^*_{\alpha_k}(n_0,i_0)}{R+\lambda+\mu}+\frac{\alpha_n u^{\beta^*_{\alpha_k}}_{\alpha_k}(n_0,i_0)}{R+\lambda+\mu}+u^{\beta^*_{\alpha_k}}_{\alpha_k}(n,i)\leq \frac{r(n,i,\tilde{\beta})}{R+\lambda+\mu}+\sum_{(m,j)}u^{\beta^*_{\alpha_k}}_{\alpha_k}(n,i)\biggl[\frac{\pi_{(n,i)(m,j)}^{\tilde{\beta}}}{R+\lambda+\mu}+\delta_{(n,i)(m,j)}\biggr]$$ for all $(n,i,\tilde{\beta})\in K$.
		Using this and (\ref{EQ8.4}), we get
		\begin{align*}
		\frac{g^*}{R+\lambda+\mu}+\bar{u}(n,i)\leq \frac{r(n,i,\tilde{\beta})}{R+\lambda+\mu}+\sum_{(m,j)}\bar{u}(m,j)\biggl[\frac{\pi_{(n,i)(m,j)}^{\tilde{\beta}}}{R+\lambda+\mu}+\delta_{(n,i)(m,j)}\biggr]
		\end{align*}
		for all $(n,i,\tilde{\beta})\in K$.
		Thus we get
		\begin{equation}\label{EN8.6}
		g^*\leq \inf_{\tilde{\beta} \in [\gamma,R]}\biggl\{r(n,i,\tilde{\beta})+\sum_{(m,j)}\bar{u}(m,j)\pi^{\tilde{\beta}}_{(n,i)(m,j)}\biggl\}.
		\end{equation}
		Now there exists $\beta_k\in \mathscr{U}_{SM}$ such that for all $(n,i)\in \mathbb{N}_0 \times \mathbb{N}_0$, we have
		\begin{align}\label{EQN8.6}
		\frac{\alpha_k I^*_{\alpha_k}(n,i_0)}{R+\lambda+\mu}+\frac{\alpha_k u^{\beta^*_{\alpha_k}}_{\alpha_k}(n,i_0)}{R+\lambda+\mu}+u^{\beta^*_{\alpha_k}}_{\alpha_k}(n,i)= \frac{r(n,i,\beta_{k})}{R+\lambda+\mu}+\sum_{(m,j)}u^{\beta^*_{\alpha_k}}_{\alpha_k}(n,i)\biggl[\frac{\pi_{(n,i)(m,j)}^{\beta_{k}}}{R+\lambda+\mu}+\delta_{(n,i)(m,j)}\biggr].
		\end{align}
		Since $\mathscr{U}_{SM}$ is compact, there exists $\beta^{'}\in \mathscr{U}_{SM}$ such that $$\lim_{k\rightarrow\infty}\beta_{k}(n,i)=\beta^{'}(n,i)~\forall (n,i)\in \mathbb{N}_0\times\mathbb{N}_0.$$ So, by the dominated convergence theorem, taking $k\rightarrow\infty$ in (\ref{EQN8.6}), we get
		$$\frac{g^*}{R+\lambda+\mu}+\bar{u}(n,i)= \frac{r(n,i,\beta^{'})}{R+\lambda+\mu}+\sum_{(m,j)}\biggl[\frac{\pi_{(n,i)(m,j)}^{\beta^{'}}}{R+\lambda+\mu}+\delta_{(n,i)(m,j)}\biggr]\bar{u}(m,j)$$ for all $(n,i)\in \mathbb{N}_0\times\mathbb{N}_0$.
		Hence we get
		\begin{align}\label{EQ8.5}
		g^*&= r(n,i,\beta^{'})+\sum_{(m,j)}\pi^{\beta^{'}}_{(n,i)(m,j)}\bar{u}(m,j)\nonumber\\
		&\geq \inf_{\tilde{\beta} \in [\gamma,R]}\biggl\{r(n,i,\tilde{\beta})+\sum_{(m,j)}\pi^{\tilde{\beta}}_{(n,i)(m,j)}\bar{u}(m,j)\biggl\}.
		\end{align}
		From (\ref{EN8.6}) and (\ref{EQ8.5}), we get (\ref{EQ8.1}).
		Now we prove that $g^{*}=J^*(n,i)$ for every $(n,i)\in \mathbb{N}_0\times\mathbb{N}_0$.
		Take an arbitrary $\beta\in \mathscr{U}_{SM}$. Then from (\ref{EQ8.1}), we get for $\beta\in\mathscr{U}_{SM}$,
		$$g^*\leq \biggl\{r(n,i,\beta)+\sum_{(m,j)}\pi^{\beta}_{(n,i)(m,j)}\bar{u}(m,j)\biggl\}~\forall (n,i)\in \mathbb{N}_0\times\mathbb{N}_0.$$ Then by [Guo and Hern$\acute{\rm a}$ndez-Lerma (2009), Proposition 7.3], we get $g^*\leq J(n,i,\beta)$. Hence $g^*\leq J^{*}(n,i)$ for every $(n,i)\in \mathbb{N}_0\times\mathbb{N}_0$. Now there exists $\beta^*\in \mathscr{U}_{SM}$ for which
		$$g^*= \biggl\{r(n,i,\beta^*)+\sum_{(m,j)}\bar{u}(m,j)\pi^{\beta^*}_{(n,i)(m,j)}\biggr\}~\forall (n,i)\in \mathbb{N}_0\times\mathbb{N}_0.$$ Hence, by [Guo and Hern$\acute{\rm a}$ndez-Lerma (2009), Proposition 7.3], we get $g^{*}=J(n,i,\beta^*)$. Hence $g^{*}=J(n,i,\beta^*)=J^{*}(n,i)$ for all $(n,i)\in \mathbb{N}_0\times\mathbb{N}_0$. Consequently, $\beta^*$ is AC-optimal.\\
		Now by [Guo and Hern$\acute{\rm a}$ndez-Lerma (2009), (7.3)], we have
		\begin{align}\label{EN8.9}
		J(n,i,\beta)=\sum_{(m,j)}r(m,j,\beta)\vartheta_\beta(m,j)=g(\beta)~\forall f\in \mathscr{U}_{SM}~\text{and}~(n,i)\in \mathbb{N}_0\times\mathbb{N}_0,
		\end{align} 
		where $g(\beta):=\sum_{(m,j)}r(m,j,\beta)\vartheta_\beta(m,j)$.\\
		Now we prove the necessary part for a determistic stationary policy to be AC optimal by contradiction. So, suppose that $\beta^{*}\in \mathscr{U}_{SM}$ is an AC optimal that dose not attain the minimum in the ACOE (\ref{EQ8.1}). Then there exist $(n^{'},i^{'})\in \mathbb{N}_0\times\mathbb{N}_0$ and a constant $d>0$ (depending on $(n^{'},i^{'})$ and $\beta^*$) such that
		\begin{align}\label{EQ8.6}
		g^*\leq r(n,i,\beta^*)-d\delta_{(n^{'},i^{'})(m,j)}+\sum_{(m,j)}\pi^{\beta^*}_{(n,i)(m,j)}\bar{u}(m,j)~\forall (n,i)\in \mathbb{N}_0\times\mathbb{N}_0.
		\end{align}
		By the irreducibility condition of the transition rates, the invariant measure $\vartheta_{\beta^*}$ of $p((n,i),t,(m,j),\beta^*)$ is supported on all of $\mathbb{N}_0\times\mathbb{N}_0$, meaning that $\vartheta_{\beta^*}(m,j)>0$ for every $(m,j)\in \mathbb{N}_0\times\mathbb{N}_0$. So, as in the proof of (\ref{EN8.9}), from (\ref{EQ8.6}) and [Guo and Hern$\acute{\rm a}$ndez-Lerma (2009), Proposition 7.3], we have
		\begin{equation}\label{EQ8.7}
		g^*\leq g(\beta^*)-d \vartheta_{\beta^*}(n^{'},i^{'})< g(\beta^*),
		\end{equation}
		which is a contradiction. So, $\beta^*$ is AC-optimal.\\
		By similar argumets as in [Guo and Hern$\acute{\rm a}$ndez-Lerma (2009), Theorem 7.8], we get the uniqueness of the solution of ACOE (\ref{EQ8.1}).
	\end{proof}
	\textbf{The Bias of a stationary policy:}
	Let $\beta\in \mathscr{U}_{SM}$. We say that a pair $(g^{'},h_\beta)\in \mathbb{R}\times B_W(\mathbb{N}_0 \times \mathbb{N}_0)$ is a solution to the Poisson equation for $\beta\in \mathscr{U}_{SM}$ if
	$$	g^{'}=r(n,i,\beta)+\sum_{(m,j)}h_\beta(m,j)\pi^{\beta}_{(n,i)(m,j)}~\forall (n,i)\in \mathbb{N}_0\times\mathbb{N}_0.$$ Define $g(\beta)=\sum_{(m,j)}r(m,j,\beta)\vartheta_\beta(m,j)$.\\
	Then by recalling [Guo and Hern$\acute{\rm a}$ndez-Lerma (2009), (7.13)], the expected average cost (loss) of $\beta$ is \begin{align}\label{EQ6.3}
	J(n,i,\beta)=\sum_{(m,j)}r(m,j,\beta)\vartheta_\beta(m,j)=g(\beta)=\vartheta_{\beta}(r(\cdot,\cdot,\beta)),~(n,i)\in \mathbb{N}_0\times\mathbb{N}_0,
	\end{align}
	see the definition of $\vartheta_\beta$ on p. 14.\\
	Next we define the bias (or ``potential''-see [Guo and Hern$\acute{\rm a}$ndez-Lerma (2009), Remark 3.2]) of $\beta\in\mathscr{U}_{SM}$ as
	\begin{align}\label{EQ6.4}
	h_\beta(n,i):=\int_{0}^{\infty}[E^\beta_{(n,i)}r(Y(t),\beta)-g(\beta)]dt~\text{ for }
	~(n,i)\in \mathbb{N}_0\times\mathbb{N}_0.
	\end{align}\\
	Next we state a Proposition whose proof is in [Guo and Hern$\acute{\rm a}$ndez-Lerma (2009), Proposition 7.11].
	\begin{proposition}\label{P2}
		Under Assumptions \ref{A5}, \ref{A2} and \ref{A3}, for every $\beta\in \mathscr{U}_{SM}$, the solution to the Poisson equation for $\beta$ are of the form $$(g(\beta),h_\beta+z)~\text{with}~z~\text{any real number}.$$ Moreover, $(g(\beta),h_\beta)$ is the unique solution to the Poisson equation
		\begin{align}\label{EQ8.8}
		g(\beta)=r(n,i,\beta)+\sum_{(m,j)}h_\beta(m,j)\pi^{\beta}_{(n,i)(m,j)}~\forall (n,i)\in \mathbb{N}_0\times\mathbb{N}_0
		\end{align}
		for which $\vartheta_\beta(h_\beta)=0$.
	\end{proposition}

	\subsection{\bf{{The Average-Cost Policy Iteration Algorithm:}}}In view of Theorem \ref{T10} given below, one can use the policy iteration algorithm for computing the optimal production rate $\beta^{*}$ that is described as follows:\\
	\textbf{\textit{The Policy Iteration Algorithm 5.1:}}\label{Al4} 
	{\bf Step 1:} \\
	Take $k=0$ and $\beta_k\in \mathscr{U}_{SM}$.\\
	{\bf Step 2:} Solve for the invariant probability measure $\	\vartheta_{\beta_k} $ from system of equations (see [Guo and Hern$\acute{\rm a}$ndez-Lerma (2009), Remark 7.12 or Proposition C.12]) \begin{align*}
	&\sum_{(n,i)}\pi^{\beta_k}_{(n,i)(m,j)}\vartheta_{\beta_k}(n,i)=0~\text{for}~(m,j)\in \mathbb{N}_0\times\mathbb{N}_0,\\
	&\sum_{(m,j)}\vartheta_{\beta_k}(m,j)=1,
	\end{align*}
	then calculate the loss, $g(\beta_k)=\sum_{(m,j)}r(m,j,\beta_k)\vartheta_{\beta_k}(m,j)$ and finally, the bias, $h_{\beta_k}$ from the system of linear equations (see Proposition \ref{P2})
	\begin{align*}
	\displaystyle\left\{ \begin{array}{ll}	& r(n,i,\beta_k)+\sum_{(m,j)}\pi^{\beta_k}_{(n,i)(m,j)}h(m,j)=g(\beta_k)\\
	&\sum_{(m,j)}h(m,j)\vartheta_{\beta_k}(m,j)=0.
	\end{array}\right.
	\end{align*}
	{\bf Step 3:} Define the new stationary policy $\beta_{k+1}$ in the following way:
	Set $\beta_{k+1}(n,i):=\beta_k(n,i)$ for all $(n,i)\in \mathbb{N}_0\times\mathbb{N}_0$ for which
	\begin{align}\label{E7.2}
	&	r(n,i,\beta_k(n,i))+\sum_{(m,j)}\pi^{\beta_k}_{(n,i)(m,j)}h_{\beta_k}(m,j)\nonumber\\
	&=\inf_{\tilde{\beta}\in [\gamma, R]} \biggl\{r(n,i,\tilde{\beta})+\sum_{(m,j)}\pi^{\tilde{\beta}}_{(n,i)(m,j)}h_{\beta_k}(m,j)\biggr\};
	\end{align}
	otherwise (i.e., when (\ref{E7.2}) does not hold), choose $\beta_{k+1}(n,i)\in [\gamma,R]$ such that
	\begin{align}\label{E7.3}
	&	r(n,i,\beta_{k+1}(n,i))+\sum_{(m,j)}\pi^{\beta_{k+1}}_{(n,i)(m,j)}h_{\beta_k}(m,j)\nonumber\\
	&=\inf_{\tilde{\beta}\in [\gamma, R]} \biggl\{r(n,i,\tilde{\beta})+\sum_{(m,j)}\pi^{\tilde{\beta}}_{(n,i)(m,j)}h_{\beta_k}(m,j)\biggr\}.
	\end{align}
	{\bf Step 4:} If $\beta_{k+1}$ satisfies (\ref{E7.2}) for all $(n,i)\in \mathbb{N}_0\times\mathbb{N}_0$, then stop because (by Theorem \ref{T10} below) $\beta_{k+1}$ is average cost (AC) (or pathwise average cost optimal (PACO)); otherwise, replace $\beta_k$ with $\beta_{k+1}$ and go back to Step 2.
	\begin{remark}
		Now we discuss how the policy iteration algorithm works.\\
		Let $\beta_0\in \mathscr{U}_{SM}$ be the initial policy in the policy iterartion algorithm (see Step 1), and let $\{\beta_k\}$ be the sequence of stationary polices obtained by the repeated application of the algorithm.\\ If $$\beta_k=\beta_{k+1}~ \text{for some}~k,$$ then it follows from Proposition \ref{P2} that the pair $(g(\beta_k),h_{\beta_k})$ is a solution to the ACOE, and thus, by Theorem \ref{T9}, $\beta_k$ is AC optimal. Hence, to analyze the convergence of the policy iterarion algorithm, we will consider the case
		\begin{equation}\label{EQ8.12}
		\beta_k\neq \beta_{k+1}~\text{for every}~k\geq 0.
		\end{equation}
		Define, for $k\geq 1$ and $(n,i)\in \mathbb{N}_0\times\mathbb{N}_0$,
		\begin{align*}
		\varepsilon(n,i,\beta_k):=&\biggl[r(n,i,\beta_{k-1})+\sum_{(m,j)}\pi^{\beta_{k-1}}_{(n,i)(m,j)}h_{\beta_{k-1}}(m,j)\biggr]\\&-\biggl[r(n,i,\beta_k)+\sum_{(m,j)}\pi^{\beta_k}_{(n,i)(m,j)}h_{\beta_{k-1}}(m,j)\biggr],
		\end{align*}
		which by Proposition \ref{P2} can be expressed as
		\begin{align}\label{EQ8.13}
		\varepsilon(n,i,\beta_k)=g(\beta_{k-1})-\biggl[r(n,i,\beta_k)+\sum_{(m,j)}\pi^{\beta_k}_{(n,i)(m,j)}h_{\beta_{k-1}}(m,j)\biggr].
		\end{align}
		Observe (by Step 3 above) that $\varepsilon(n,i,\beta_k)=0$ if $\beta_k(n,i)=\beta_{k-1}(n,i)$, whereas $\varepsilon(n,i,\beta_k)>0$ if $\beta_k(n,i)\neq \beta_{k-1}(n,i)$.\\ Hence, $\varepsilon(n,i,\beta_k)$ can be interpreted as the ``improvement'' of the nth iteration of the algorithm.
	\end{remark}
	
	\subsection*{Numerical Example:} Figure \ref{figure4} shows the results obtained from the above algorithm using the same parameters as in the previous experiments with the exception of $\alpha$, see p. 18-19. It can be noticed that the optimal production rates for some states such as $(2,1), (3,2)$ and $(4,2)$ are different from that of discounted cost algorithms. Discounted cost criterion assumes more importance to the current observed state, whereas average cost assumes equal importance to each state observed at transitions. At states $(2,1), (3,2)$ and $(4,2)$, the optimal policy advises higher production rate in consideration of future states as compared to the discounted cost criterion. 

\begin{figure}[H]
	\centering
	\includegraphics[width=0.6\textwidth]{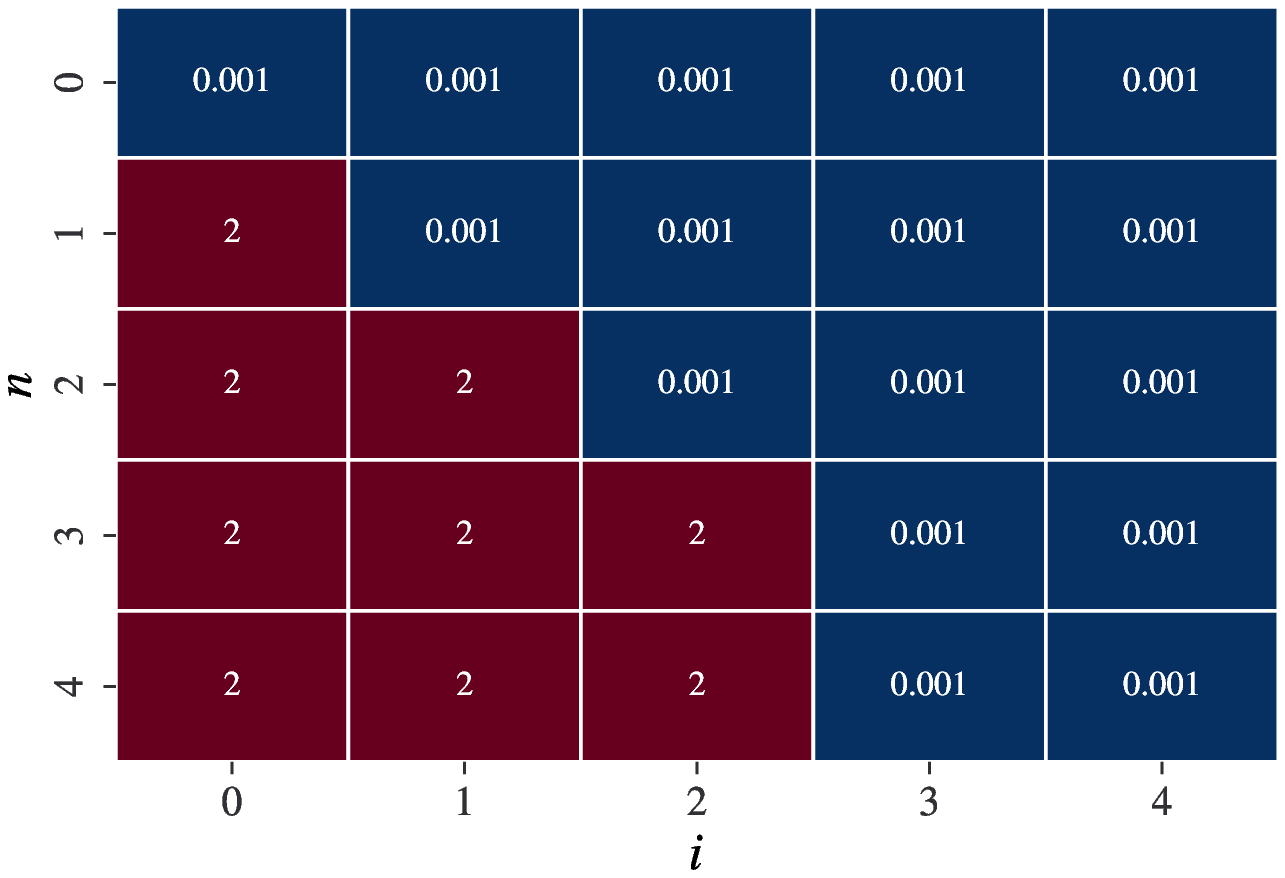}
	\caption{Optimal policy corresponding to average-cost policy iteration algorithm. }
	\label{figure4}
\end{figure}

	Next by [Guo and Hern$\acute{\rm a}$ndez-Lerma (2009), Lemma 7.13], we have the following Lemma.
	\begin{lemma}\label{L11}
		Under Assumptions \ref{A5}, \ref{A2} and \ref{A3}, suppose that (\ref{EQ8.12}) is satisfied. Then the following statements hold
		\begin{itemize}
			\item [(a)] The sequence $\{g(\beta_k)\}$ is strictly decreasing and it has a finite limit.\\
			\item [(b)] For every $(n,i)\in \mathbb{N}_0\times\mathbb{N}_0$, $\varepsilon(n,i,\beta_k)\rightarrow 0$ as $k\rightarrow \infty$.
		\end{itemize}
	\end{lemma}
	To get the optimal policy from the policy iteration algorithm 5.1, we prove the following Theorem.
	\begin{thm}\label{T10}
		Suppose that Assumptions \ref{A5}, \ref{A2} and \ref{A3} hold, and let $\beta_1\in \mathscr{U}_{SM}$ be an arbitrary initial policy for the policy iteration algorithm 5.1. Let $\{\beta_k\}\subset \mathscr{U}_{SM}$ be the sequence of polices obtained by the policy iteration algorithm 5.1. Then one of the following results hold.
		\begin{itemize}
			\item [(a)] Either 
			\begin{itemize}
				\item [(i)] the algorithm converges in a finite number of iterations to an AC optimal policy;\\
				Or\\
				\item [(ii)] as $k\rightarrow\infty$, the sequnce $\{g(\beta_k)\}$ converges to the optimal AC function value $g^*$, and any limit point of $\{\beta_k\}$ is an AC optimal stationary policy.
			\end{itemize}
			\item [(b)] There exists a subsequence $\{\beta_l\}\subset \{\beta_k\}$ for which 
			\begin{align*}
			&g(\beta_l)\rightarrow g~\text{[Lemma~\ref{L11}]},~\beta_l\rightarrow \beta,~\text{and}\nonumber\\
			&h_{\beta_l}\rightarrow h~\text{[pointwise]}.
			\end{align*}
			In addition, the limiting triplet $(g,\beta,h)\in \mathbb{R}\times\mathscr{U}_{SM}\times B_W(\mathbb{N}_0 \times \mathbb{N}_0)$ satisfies 
			\begin{align}\label{E6.6}
			g&=r(n,i,\beta)+\sum_{(m,j)}\pi^{\beta}_{(n,i)(m,j)}h(m,j)\nonumber\\
			&=\inf_{\tilde{\beta}\in [\gamma, R]}\biggl\{r(n,i,\tilde{\beta})+\sum_{(m,j)}\pi^{\tilde{\beta}}_{(n,i)(m,j)}h(m,j)\biggr\}.
			\end{align}
		\end{itemize}
	\end{thm}
	\begin{proof}
		Note that it is enough to prove part (b).\\
		Let $\{\beta_k\}$ satisfy (\ref{EQ8.12}).
		In view of Lemma \ref{L11}, since $\mathscr{U}_{SM}$ is compact and $h_{\beta_l}\in B_W(\mathbb{N}_0 \times \mathbb{N}_0)$, there exists a subsequence $\{\beta_l\}$ of $\beta_k$ such that $h_{\beta_l}$ converges pointwise to some $h\in B_W(\mathbb{N}_0 \times \mathbb{N}_0)$. So, we have
		\begin{align}\label{EQ8.14}
		&g(\beta_l)\rightarrow g,~\beta_l\rightarrow \beta,~\text{and}\nonumber\\
		&h_{\beta_l}\rightarrow h.
		\end{align}
		Now, by Proposition \ref{P2} and the definition of the improvement term $\varepsilon(n,i,\beta_{l+1})$  in (\ref{EQ8.13}), we have
		\begin{align}\label{EQ8.15}
		g(\beta_{l})&=\biggl[r(n,i,\beta_l)+\sum_{(m,j)}\pi^{\beta_l}_{(n,i)(m,j)}h_{\beta_{l}}(m,j)\biggl]\nonumber\\
		&=\min_{\tilde{\beta}\in [\gamma, R]}\biggl\{r(n,i,\tilde{\beta})+\sum_{(m,j)}\pi^{\tilde{\beta}}_{(n,i)(m,j)}h{(m,j)}\biggr\}+\varepsilon(n,i,\beta_{l+1}).
		\end{align}
		Taking $l\rightarrow\infty$, we have
		\begin{align}\label{EQ8.16}
		g&=\biggl[r(n,i,\beta)+\sum_{(m,j)}\pi^{\beta}_{(n,i)(m,j)}h_{\beta}(m,j)\biggr]\nonumber\\
		&=\inf_{\tilde{\beta}\in [\gamma, R]}\biggl\{r(n,i,\tilde{\beta})+\sum_{(m,j)}\pi^{\tilde{\beta}}_{(n,i)(m,j)}h{(m,j)}\biggr\}~\forall (n,i)\in \mathbb{N}_0 \times \mathbb{N}_0.
		\end{align}
		Hence $\beta$ is AC optimal and $g$ is optimal AC function.
	\end{proof}
	\section{Average optimality for pathwise costs:} 
	In section 5, we have studied the optimality problem under the expected average cost $J(n,i,\beta)$. However, the sample-path reward $r(Y(t), \beta)$ corresponding to an average-reward optimal policy that minimizes an expected average cost may have fluctuations from its expected value. To take these fluctuations into account, we next consider the pathwise average-cost (PAC) criterion.\\
	In the next theorem, we find the existence of solution of the pathwise average cost optimality equation (PACOE).\\
	Here we give an outline of proof of the following optimality Theorem; for details, see [Guo and Hern$\acute{\rm a}$ndez-Lerma (2009), Theorem 8.5].
	\begin{thm}\label{T11}
		Under Assumptions \ref{A5}, \ref{A2}, \ref{A3} and \ref{A4}, the following statements hold.
		\begin{itemize}
			\item [(a)] There exist a unique $g^*$, a function $u^*\in B_W(\mathbb{N}_0\times\mathbb{N}_0)$, and a stationary policy $\beta^*\in \mathscr{U}_{SM}$ satisfying the average-cost optimality equation (ACOE)
			\begin{align}\label{EQ9.4}
			g^*&=\inf_{\tilde{\beta}\in [\gamma,R]}\biggl\{r(n,i,\tilde{\beta})+\sum_{(m,j)}u(m,j)\pi_{(n,i)(m,j)}^{\tilde{\beta}}\biggr\}\nonumber\\
			&=r(n,i,\beta^*(n,i))+\sum_{(m,j)}u(m,j)\pi_{(n,i)(m,j)}^{\beta^*}~\forall (n,i)\in \mathbb{N}_0\times\mathbb{N}_0.
			\end{align}
			\item [(b)] The policy $\beta^*$ in (a) is PAC-optimal, and $P^{\beta^*}_{(n,i)}(J_c(n,i,\beta^*)=g^*)=1$ for all $(n,i)\in \mathbb{N}_0\times\mathbb{N}_0$, with $g^*$ as in (a).\\
			\item [(c)] A policy in $\mathscr{U}_{SM}$ is PAC-optimal iff it realizes the minimum in (\ref{EQ9.4}).\\
			\item [(d)] For given $\varepsilon\geq 0$, $\beta\in \mathscr{U}_{SM}$, and $g^*$ as in (a) above, if there is a function $u^{'}\in B_W(\mathbb{N}_0\times\mathbb{N}_0)$ such that
			\begin{align}\label{EQ9.5}
			g^*\geq r(n,i,\beta(n,i))+\sum_{(m,j)}u^{'}(m,j)\pi_{(n,i)(m,j)}^{\beta}-\varepsilon~\forall (n,i)\in \mathbb{N}_0\times\mathbb{N}_0,
			\end{align}
			then $\beta$ is $\varepsilon$-PAC-optimal.
		\end{itemize}
	\end{thm}
	\begin{proof}
		\begin{itemize}
			\item [(a)] Note that, part (a) has been obtained in Theorem \ref{T9}, see [Guo and Hern$\acute{\rm a}$ndez-Lerma (2009), Remark 8.4]. The proof is based on the fact that if $\beta_k$, $g(\beta_k)$, and $h_{\beta_k}$ are as in (\ref{E7.2})-\eqref{EQ8.13}, then there exist a subsequence $\{\beta_{k_l}\}$ of $\beta_k$, $\beta^*\in \mathscr{U}_{SM}$, $u^*\in B_W(\mathbb{N}_0\times\mathbb{N}_0)$, and a constant $g^*$ such that for each $(n,i)\in \mathbb{N}_0\times\mathbb{N}_0$,
			\begin{align}\label{EQ9.6}
			\lim_{l\rightarrow\infty}h_{\beta_{k_l}}=:u^*,~	\lim_{l\rightarrow\infty} \beta_{k_l}=\beta^*,~\text{and}~	\lim_{l\rightarrow\infty}g(\beta_{k_l})=g^*.
			\end{align}
			The triplet $(g^*,u^*,\beta^*)$ satisfies (\ref{EQ9.4}).
			\item [(b)] To prove (b), for all $(n,i)\in \mathbb{N}_0\times\mathbb{N}_0$, $\beta\in \mathscr{U}_{SM}$ let
			\begin{align}
			&\Delta(n,i,\beta(n,i)):=r(n,i,\beta(n,i))+\sum_{(m,j)}u^*(m,j)\pi_{(n,i)(m,j)}^{\beta}-g^*,\label{EQ9.7}\\
			&\bar{h}(n,i,\beta):=\sum_{(m,j)}u^*(m,j)\pi^{\beta}_{(n,i)(m,j)}.\label{EQN9.7}
			\end{align}
			We define the (continuous-time) stochastic process
			\begin{equation}\label{EQ9.8}
			M(t,\beta):=\int_{0}^{t}\bar{h}(Y(y),\beta)dy-u^*(Y(t))~\text{ for }~t\geq 0.
			\end{equation}
			By similar arguments as in [Guo and Hern$\acute{\rm a}$ndez-Lerma (2009), Theorem 8.5], we have
			\begin{equation}\label{EQ9.9}
			M(t,\beta)=-\int_{0}^{t}r(Y(y),\beta)dy+\int_{0}^{t}\Delta(Y(y),\beta)dy-u^*(Y(t))+t g^*.
			\end{equation}
			Then from \eqref{E2.2}, \eqref{EQ9.4} and (\ref{EQ9.7}) we get $\Delta(n,i,\beta)\leq 0$ and $\Delta(n,i,\beta
			^*)=0$ for all $(n,i)\in \mathbb{N}_0\times\mathbb{N}_0$.
			Thus by [Guo and Hern$\acute{\rm a}$ndez-Lerma (2009), Theorem 8.5, equations (8.31), (8.32)] and \eqref{EQ9.9}, we get
			\begin{align}\label{EQ10}
			&P^{\beta}_{(n,i)}(J_c(n,i,\beta)\geq g^*)=1~\text{and}\nonumber\\
			&P^{\beta^*}_{(n,i)}(J_c(n,i,\beta^*)=g^*)=1.
			\end{align}
			Since $\beta\in \mathscr{U}_{SM}$ and $(n,i)\in\mathbb{N}_0\times\mathbb{N}_0$ are arbitrary, we get part (b).
			\item [(c)] See Theorem \ref{T9} (b) or [Guo and Hern$\acute{\rm a}$ndez-Lerma (2009), Theorem 8.5 (c)].
			\item [(d)] Let $\Delta_{u^{'}}(n,i,\beta(n,i)):=r(n,i,\beta(n,i))+\sum_{(m,j)}u^{'}(m,j)\pi^{\beta}_{(n,i)(m,j)}-g^*$. Then by (\ref{EQ9.5}), we have $\Delta_{u^{'}}(n,i,\beta(n,i))\leq -\varepsilon$ for all $(n,i)\in \mathbb{N}_0\times\mathbb{N}_0$. So, (as proof of \eqref{EQ10}) we have
			$$P^{\beta}_{(n,i)}(J_c(n,i,\beta)\leq g^*+\varepsilon)=1,$$ which together with (b), gives (d).
		\end{itemize}
	\end{proof}
	\subsection{{\bf The Pathwise Average-Cost Policy Iteration Algorithm:}}
	In view of Theorem \ref{T11} and Proposition \ref{P4} given below, one can use the policy iteration algorithm for computing the optimal production rate $\beta^{*}$.
 To compute this optimal production rate $\beta^*$, we describe the following policy iteration algorithm.\\ \noindent	\textbf{\textit{The Policy Iteration Algorithm} 6.1:}\label{Al5}
	See the Policy Iteration Algorithm 5.1 for computing the optimal production rate $\beta^{*}$.\\
	Next in view of Theorem \ref{T11} (c), we have the following Proposition.
	\begin{proposition}\label{P4}
		Suppose that Assumptions \ref{A5}, \ref{A2}, \ref{A3} and \ref{A4} hold. Then any limit point $\beta^*$ of the sequence $\{\beta_k\}$ obtained by the Policy Iteration Algorithm 5.1 is PAC-optimal.
	\end{proposition}

	\section{Conclusions} 
	In this article, we have examined a production inventory dynamic control system for discounted, average, and pathwise average cost criterion for risk-neutral cost (i.e., expectation of the total cost) criterion. Here, the demands arrive at the production workshop according to a Poisson process and the processing time of the customer's demands is exponentially distributed. Each production is one unit and the production is kept running until the inventory level becomes sufficiently large, and the production is on a make-to-order basis. We assume that the production time of an item follows an exponential distribution and that the amount of time for the item produced to reach the retail shop is negligible. In addition, we have assumed that no new customer join the queue when there is a void inventory. This yields an explicit product-form solution for the steady-state probability vector of the system. We further discuss the policy and value iteration algorithms for each cost criterion. Using these algorithms, we obtain the optimal production rate that minimizes the discounted/average/pathwise average total cost per production using a Markov decision process approach. Numerical examples are used to verify the discussed algorithms. 
		
	Our proposed model with service time, production time, or lead time following a general distribution can be direct extensions of this work. Another potential research would be to conduct the same analysis under the risk-sensitive utility (i.e., expectation of the exponential of the total cost) cost criterion, which provides more comprehensive protection from risk than the risk-neutral case.

	~\\ \textbf{Acknowledgment}
	
 The research work of Chandan Pal is partially supported by SERB, India, grant MTR/2021/000307 and the research work of Manikandan, R., is supported by DST-RSF  research project no. 64800 (DST) and research project no. 22-49-02023 (RSF).

\end{document}